\documentclass[12pt]{amsart}

\usepackage[margin=1.3in]{geometry}
\usepackage{amssymb,amsmath}
\usepackage{amsfonts}
\usepackage{amsthm}
\usepackage{xcolor}
\usepackage{hyperref}
\usepackage{mathrsfs}
\usepackage{indentfirst}
\usepackage{enumitem}

\newtheorem{thm}{Theorem}[section]
\newtheorem{lem}[thm]{Lemma}
\newtheorem{prop}[thm]{Proposition}

\newtheorem{rem}[thm]{Remark}

\numberwithin{equation}{section}

\def\XXint#1#2#3{{\setbox0=\hbox{$#1{#2#3}{\int}$}
		\vcenter{\hbox{$#2#3$}}\kern-.5\wd0}}

\newcommand{\R}{\mathbb{R}}

\newcommand{\osc}{\operatorname*{osc}}

\begin{document}
		\title[Optimal Asymptotics for Supercritical LMC Equations]
		{\textbf{Optimal Asymptotic Behavior at Infinity for Solutions of the Supercritical Lagrangian Mean Curvature Equation in Exterior Domains}}

\author[J.~Bao]{Jiguang Bao}
\thanks{The first author is supported by the National Natural Science Foundation of China (12371200) and Beijing Natural
	Science Foundation (1254049).}
\author[Q.~Jiang]{Qinfeng Jiang*}
\thanks{*Corresponding author.}

\subjclass[2020]{Primary 35J60; Secondary 35B40, 35J96, 35C20}
\keywords{Lagrangian mean curvature equation, supercritical phase, exterior domains, optimal asymptotic expansion, Lipschitz perturbation}

\maketitle

\begin{center}
\normalsize
School of Mathematical Sciences, Beijing Normal University,\\
Beijing, 100875, China\\[0.3em]
\texttt{jgbao@bnu.edu.cn}\\
\texttt{202531130031@mail.bnu.edu.cn}
\end{center}

\begin{center}
\begin{minipage}{0.88\textwidth}
\small
\noindent\textbf{Abstract.}
We study the asymptotic behavior at infinity of solutions to the supercritical
Lagrangian mean curvature equation
\[
\sum_{i=1}^n \arctan \lambda_i(D^2u)=\theta+f(x)
\]
on exterior domains in \(\mathbb R^n\), \(n\ge 2\), where
\(|\theta|>((n-2)\pi)/2\). The perturbation \(f\) is
assumed to be locally Lipschitz near infinity and to satisfy a decay condition with rate \(\beta>0\). The main new ingredient is a
scale-dependent difference quotient argument, combined with a nonlocal
potential method, which avoids differentiating \(f\) twice and yields
quantitative Hessian convergence under only Lipschitz regularity. We establish
optimal asymptotic expansions in all dimensions and for all decay rates
\(\beta>0\), including the critical logarithmic cases. This improves previous
results requiring higher regularity of \(f\) and faster decay in \cite{BJ2026}.
\end{minipage}
\end{center}

\section{Introduction}\label{sec1}
In this paper, we study the asymptotic behavior at infinity of solutions to the
Lagrangian mean curvature equation
\begin{equation}\label{eq}
	F(D^2u):=\sum_{i=1}^n \arctan \lambda_i(D^2u)=\theta+f(x)
	\qquad \text{in } \mathbb R^n.
\end{equation}
Here $\lambda_i(D^2u)$, $i=1,2,\cdots,n$, denote the eigenvalues of the Hessian $D^2u$, 
\(|\theta|>((n-2)\pi)/2\), and $f$ is a perturbation term.

Within the framework of calibrated geometry, Harvey and Lawson \cite{HL1982} first introduced the special Lagrangian equation 
\begin{equation}\label{eq1}
	F(D^2u) := \sum_{i=1}^{n} \arctan \lambda_i(D^2u) = \theta, 
\end{equation}
where $\theta \in \left(-n\pi/2,n\pi/2\right)$ is a constant. 
The left-hand side of equation~\eqref{eq1} represents the argument of the complex number $(1+\sqrt{-1}\lambda_{1}(D^{2}u))\cdots(1+\sqrt{-1}\lambda_{n}(D^{2}u))$, which is usually called the Lagrangian phase. Its solutions $u$ were shown to have the property that the gradient graph $(x, Du(x))$
in Euclidean space is a Lagrangian submanifold that is absolutely volume-minimizing. 

The arctangent operator
is clearly elliptic for any function $u$. Because $\arctan$ is an odd function, we may assume without loss of generality that $\theta \geq 0$. In the literature \cite{Yuan2006}, the Lagrangian phase $(n - 2)\pi/2$ is usually called critical because the level set
\[L_\theta:=\left\{\lambda=\left(\lambda_1,\cdots,\lambda_n\right)\in\mathbb{R}^n\vert\sum_{i=1}^n\arctan\lambda_i=\theta\right\}\]
is convex only if $\theta\geq (n - 2)\pi/2$. For $\theta \ge (n-1)\pi/2$, convexity of $L_{\theta}$ is immediate, because $\lambda_i \geq 0$ for all $i$ and the operator $F$ becomes concave. For supercritical phases $\theta > (n-2)\pi/2 $, the operator $F$ can be extended to a concave operator \cite{CPW2017,CW2019}. This structural property plays a fundamental role in the analysis of solutions, particularly in establishing regularity results.

There are several rigidity theorems for the special Lagrangian equation~\eqref{eq1} on the whole space. Jost--Xin \cite[Theorem~1]{JX2002} proved that any entire smooth convex solution with bounded Hessian must be a quadratic polynomial; see also \cite[Theorem~7.3.4]{Xin2003}. Yuan \cite{Y2002,Yuan2006} proved Bernstein-type theorems for entire smooth solutions under the supercritical phase condition and, separately, under semiconvexity assumptions. Related Liouville-type results under additional structural conditions were obtained in \cite{WY2008,D2023}. Chen--Warren--Yuan's interior regularity theory \cite{CWY2009} showed that convex viscosity solutions are smooth when $\theta\geq (n-1)\pi/2$, and a subsequent result \cite{CSY2023} established real analyticity for all convex viscosity solutions. The strict supercritical phase condition in the Bernstein theorem cannot in general be weakened to the critical phase. Indeed, Warren \cite{W2010} constructed the non-quadratic entire solution
\[
u(x)=\left(x_1^2+x_2^2-1\right)e^{-x_3}+\frac14 e^{x_3}
\]
to \eqref{eq1} with $\theta=\pi/2$ in $\mathbb R^3$.

Li--Li--Yuan \cite{LLY2020} studied solutions to \eqref{eq1} in exterior domains and established asymptotic expansions for solutions near infinity. They showed that any smooth solution $u$ of \eqref{eq1} with supercritical phase or under a semiconvexity condition must asymptotically approach a quadratic polynomial (with an additional logarithmic term when $n=2$). Specifically, for dimensions $n\geq3$, they proved the existence of a symmetric
\(n\times n\) matrix \(A\), a vector \(b\in\mathbb R^n\), and a constant \(c\in\mathbb R\) such that
\begin{equation*}
	u(x)= \frac{1}{2}x^\mathsf{T} A x + b^\mathsf{T} x + c+O_k(|x|^{2-n})\quad \text{as}~|x|\to +\infty.
\end{equation*}
For $n=2$, they proved the existence of a symmetric \(2\times2\) matrix \(A\), a vector $b \in \mathbb{R}^2$, and constants $c,d \in \mathbb{R}$ such that
\begin{equation}\label{fd5}
	u(x)= \frac{1}{2}x^\mathsf{T} A x + b^\mathsf{T} x +\frac{d}{2}\ln\left(x^{\mathsf{T}}(I+A^2)x\right) + c+O_k(|x|^{-1})\quad \text{as}~|x|\to +\infty,
\end{equation}
for all $k\in \mathbb{N}$. The notation $\phi(x)=O_k\left(|x|^{k_1}(\ln |x|)^{k_2}\right)$, $k_1\leq 0, k_2\geq 0$, means that for $t=0,\cdots,k$,
\[|D^t\phi(x)|=O(|x|^{k_1-t}(\ln |x|)^{k_2}),\quad \text{as}~|x|\to+\infty.\]

Subsequent works by Liu--Bao \cite{LB2022JGA,LB2023} derived higher-order asymptotic expansions. Recently, Han--Marchenko \cite{HM2025} also used a single function $v$ to characterize remainders in the asymptotic expansions via a modified Kelvin transform. When the constant right-hand side \(\theta\) in \eqref{eq1} is replaced by a
function \(\theta+f(x)\) with \(f(x)\to0\) at infinity, Liu--Bao \cite{LB2022} established the asymptotic expansion at infinity in dimension two under the assumption of quadratic growth.

Recall that, without loss of generality, we assume $\theta\geq 0$. Let
\[\mathcal{A}:=\{M \in \mathcal{S}^{n\times n}\mid \sum_{i=1}^{n} \arctan\lambda_i(M)=\theta\},\]
where $\mathcal{S}^{n\times n}$ denotes the linear space of symmetric $n\times n$ real matrices. Throughout the paper, we assume $\Omega\subset\mathbb R^n$ is a bounded open set.

In \cite{BJ2026}, the authors studied convergence rates and required \(f\) to be at least \(C^3\), since they employed the iteration method developed in \cite{BLZ2015}. They obtained the following result.
\begin{thm}[{\cite[Theorems 1.1 and 1.2]{BJ2026}}]\label{D2ueo}
	Let $n=2$, $\theta>0$, and let $u$ be a smooth solution of~\eqref{eq}. Assume that $f\in C^{m}(\mathbb{R}^2\setminus \overline{\Omega})$, $m\geq 3$. If there exists a positive constant $\beta > 2$ such that
	\begin{equation}\label{Ducondo}
	\limsup_{|x|\to\infty} \left(|x|+|Du(x)|\right)^{\beta+l}\bigl| D^{l} f(x) \bigr|<\infty,	
	\end{equation}
	for $l = 0,1,2,3$, then there exist $A \in \mathcal{A}$, $b \in \mathbb{R}^2$, and $c, d \in \mathbb{R}$ such that
	\begin{equation}\label{bg2o}
		u(x)=\frac{1}{2}x^\mathsf{T} Ax+b^\mathsf{T} x+\frac{d}{2}\ln\left(x^{\mathsf{T}}(I+A^2)x\right)+c+O_k\left(|x|^{-\min \left\{1,\beta-2 \right\} } (\ln |x|)^{\mu_1}\right)
	\end{equation}
	as $|x|\to+\infty$, for all $k=0,\cdots,m+1$. Here $\mu_1 = 
	\begin{cases} 
		0, & \beta \neq 3, \\ 
		1, & \beta = 3.
	\end{cases}$
	$d$ in \eqref{bg2o} is given by
	\begin{equation}\label{d}
		d=\frac{\sqrt{\det(I+A^2)}}{2\pi}
		\cdot\int_{\mathbb{R}^2}\left(\left(I+A^2\right)_{ij}D_{ij}u(x)-\operatorname{tr} A\right)\mathrm{d}x.
	\end{equation}
	
	If $n\geq 3$, $\theta>(n-2)\pi/2$, and $u$ is a smooth solution of~\eqref{eq}. Assume that $f\in C^{m}(\mathbb{R}^n\setminus \overline{\Omega})$, $m\geq 2$. If there exists a positive constant $\beta > 2$ such that \eqref{Ducondo} holds for $l=0,1,2$, then there exist $A \in \mathcal{A}$, $b \in \mathbb{R}^n$, and $c\in \mathbb{R}$ such that
	\begin{equation}\label{bg3o}
		u(x)=\frac{1}{2}x^\mathsf{T} Ax+b^\mathsf{T} x+c+O_k\left(|x|^{2-\min \left\{\beta,n \right\} } (\ln |x|)^{\mu_2}\right)
	\end{equation}
	as $|x|\to+\infty$, for all $k=0,\cdots,m+1$. Here $\mu_2 = 
	\begin{cases} 
		0, & \beta \neq n, \\ 
		1, & \beta = n.
	\end{cases}$
\end{thm}

A natural question is whether the regularity assumption on $f$ can be weakened, as in \cite{QB2025}. Owing to the recent interior Hessian estimates of Zhou \cite{Z2025}, we are able to lower the regularity requirement to $C^{0,1}$. In contrast to the asymptotic behavior results for the Monge--Amp\`ere equation obtained in \cite{QB2025,QB20252n}, the present setting requires the phase to be Lipschitz continuous, rather than merely H\"older continuous. Indeed, there exist viscosity solutions with H\"older continuous phase that fail to be $C^2$; see \cite{AB2021,AB2024}.

Our main results are the following two theorems. The first one gives the complete
optimal asymptotic behavior in dimension two. It contains both the fast convergence
case \(\beta>2\) and the slow convergence case \(0<\beta\le2\).

We assume that
\[
f\in C^0(\mathbb R^2)\cap C^{0,1}(\mathbb R^2\setminus\overline{\Omega})
\]
and that
\begin{equation}\label{Ducond}
\limsup_{|x|\to\infty} \left(|x|+|Du(x)|\right)^{\beta}|f(x)|+
\left(|x|+|Du(x)|\right)^{\beta+1}[f]_{C^{0,1}(B_{|x|/2}(x))}<\infty.
\end{equation}
Here and below, \([\cdot]_{C^{0,\alpha}(U)}\), \(0<\alpha\le1\), denotes the
H\"older seminorm on a domain \(U\).

For $A\in \mathcal A$, $b\in \mathbb{R}^2$, \(c,d\in\mathbb R\), we define
\[P_1(x)=
\begin{cases}
	\frac12x^\mathsf{T}Ax, & 0<\beta\le1, \\
	\frac12x^\mathsf{T}Ax+b^\mathsf{T}x, & 1<\beta\le2,\\
	\frac12x^\mathsf{T}Ax+b^\mathsf{T}x+\frac d2\ln\left(x^\mathsf{T}(I+A^2)x\right)+c, & \beta>2.
\end{cases}
\]
We obtain the following asymptotic behavior:
\begin{thm}\label{D2ue}
Let \(n=2\), \(\theta>0\), and let \(u\) be a viscosity solution of~\eqref{eq}. Assume that \(f\) satisfies \eqref{Ducond} for some
\(\beta>0\). If \(0<\beta\le1\), assume that \(f\) satisfies \eqref{Ducond} without the \(Du\)-term and additionally assume that \(u\) has a quadratic upper growth bound at infinity. Then for every \(\alpha\in(0,1)\), \(u\in C^{2,\alpha}_{\mathrm{loc}}
(\mathbb R^2\setminus\overline\Omega)\), and there exist \(A\in\mathcal A\), \(b\in\mathbb R^2\), 
\(c,d\in\mathbb R\) such that
\begin{equation}\label{bg2}
	u(x)-P_1(x)=\begin{cases}
		O_2\left(|x|^{2-\beta}(\ln|x|)^{[\beta]}\right), & 0<\beta\le1, \\
		O_2\left(|x|^{2-\beta}(\ln|x|)^{2([\beta]-1)}\right), & 1<\beta\le2,\\
		O_2\left(|x|^{-\min \left\{1,\beta-2 \right\}}(\ln|x|)^{\mu_1}\right), & \beta>2,
	\end{cases}
\end{equation}
as $|x|\to +\infty$. The constant $d$ is given by \eqref{d}, and $\mu_1$ is defined as above. Here \([\beta]\) denotes the integer part of \(\beta\).

Moreover,
\[\left[D^2\left(u(x)-P_1(x)\right)\right]_{C^{0,\alpha}(B_{|x|/2}(x))}=\begin{cases}
	O\left(|x|^{-\beta-\alpha}(\ln|x|)^{[\beta]}\right), & 0<\beta\le1, \\
	O\left(|x|^{-\beta-\alpha}(\ln|x|)^{2([\beta]-1)}\right), & 1<\beta\le2,\\
	O\left(|x|^{-\sigma-\alpha}(\ln|x|)^{\mu_1}\right), & \beta>2,
\end{cases}
\]
as $|x|\to +\infty$.
\end{thm}

\begin{rem}
	We claim that all the asymptotic estimates in Theorem~\ref{D2ue} are optimal. More precisely, under assumption~\eqref{Ducond}, the remainder terms in \eqref{bg2} are sharp in the sense that they cannot, in general, be replaced by terms of strictly smaller order.
	
	When $\beta = 1$, optimality can be seen by considering the solution
	\[
	u_0(x)=\frac{1}{2}\tan\frac{\theta}{2}\,|x|^2+g_0(x),
	\]
	where $g_0$ is defined by
	\[
	g_0(x):=
	\begin{cases}
		\text{smooth}, & 0\le |x|\le 2,\\
		(x_1+x_2)\ln|x|, & |x|>2.
	\end{cases}
	\]
	A direct computation shows that
	\[
		\begin{aligned}
		\arctan\lambda_1(D^2u_0(x))+\arctan\lambda_2(D^2u_0(x))
		&=\theta+f_0(x) \\
		&=\theta+O_1(|x|^{-1}),
		\qquad \text{as } |x|\to+\infty.
		\end{aligned}
		\]
	
	When $\beta\in (0,2]$, we consider the radially symmetric function $f_1$ defined by
	\[
	f_1(x):=
	\begin{cases}
		\text{smooth}, & 0\le |x|\le 2,\\
		|x|^{-\beta}, & |x|>2.
	\end{cases}
	\]
	According to the computations in \cite{LB2023ANS,BLW2024}, the solution of \eqref{eq} with right-hand side $f_1$ has the asymptotic behavior
	\[
	u_1(x)=\frac{1}{2}\tan\frac{\theta}{2}|x|^2+
	\begin{cases}
		O_1(|x|^{2-\beta}), & \beta\in(0,2),\\
		O_1\bigl((\ln|x|)^2\bigr), & \beta=2,
	\end{cases}
	\qquad \text{as } |x|\to+\infty.
	\]
	
	To illustrate the optimality of \eqref{bg2} when $2<\beta<3$, consider
	\[
	u_2(x)=\frac{1}{2}\tan\frac{\theta}{2}|x|^2+d\ln|x|+\frac{d}{2}\ln\bigl(A_*^2+1\bigr)+c+g_2(x),
	\]
	 and
	\[
	g_2(x):=
	\begin{cases}
		\text{smooth}, & |x|\le 2,\\
		|x|^{2-\beta}, & |x|>2.
	\end{cases}
	\]
	A direct computation yields
	\[
	\arctan\lambda_1(D^2u_2)+\arctan\lambda_2(D^2u_2)
	=\theta+f_2(x)
	=\theta+O_1(|x|^{-\beta}),
	\qquad \text{as } |x|\to+\infty.
	\]
	
	For $\beta=3$, the logarithmic factor is also optimal. Indeed, consider
	\[
	u_3(x)=\frac{1}{2}\tan\frac{\theta}{2}|x|^2+d\ln|x|+\frac{d}{2}\ln\bigl(A_*^2+1\bigr)+c+g_3(x),
	\]
	where
	\[
	g_3(x):=
	\begin{cases}
		\text{smooth}, & |x|\le 2,\\[4pt]
		\displaystyle \frac{(x_1+x_2)\ln|x|}{|x|^2}, & |x|>2,
	\end{cases}
	\]
	and $g_3(x)=O\bigl(|x|^{-1}\ln|x|\bigr)$, as $|x|\to+\infty$.
	A direct computation yields
	\[
		\begin{aligned}
		\arctan \lambda_1(D^2u_3(x))+\arctan \lambda_2(D^2u_3(x))
		&=\theta+f_3(x) \\
		&=\theta+O_1(|x|^{-3}),
		\qquad \text{as } |x|\to+\infty.
		\end{aligned}
		\]
	Moreover, one can verify that the above $f_i$ and $u_i$, $i=0,1,2,3$, satisfy \eqref{Ducond}.
	
	Finally, when $\beta>3$, the optimality follows from \eqref{fd5}.
\end{rem}

Similarly, for $A\in \mathcal A$, $b\in \mathbb{R}^n$, \(c\in\mathbb R\), we define
\[P_2(x)=
\begin{cases}
	\frac12x^\mathsf{T}Ax, & 0<\beta\le1, \\
	\frac12x^\mathsf{T}Ax+b^\mathsf{T}x, & 1<\beta\le2,\\
	\frac12x^\mathsf{T}Ax+b^\mathsf{T}x+c, & \beta>2.
\end{cases}
\]

We next state the higher-dimensional result. 
\begin{thm}\label{D2uehd}
Let \(n\ge3\), \(\theta>(n-2)\pi/2\), and let \(u\) be a viscosity solution of~\eqref{eq}. Assume that \(f\) satisfies \eqref{Ducond} for some
\(\beta>0\). If \(0<\beta\le1\), assume that \(f\) satisfies \eqref{Ducond} without the \(Du\)-term and additionally assume that \(u\) has a quadratic upper growth bound at infinity. Then for every \(\alpha\in(0,1)\), \(u\in C^{2,\alpha}_{\mathrm{loc}}
(\mathbb R^n\setminus\overline\Omega)\), and there exist \(A\in\mathcal A\), \(b\in\mathbb R^n\), 
\(c\in\mathbb R\) such that
\begin{equation}\label{bgn}
	u(x)-P_2(x)=\begin{cases}
		O_2\left(|x|^{2-\beta}(\ln|x|)^{[\beta]}\right), & 0<\beta\le1, \\
		O_2\left(|x|^{2-\beta}(\ln|x|)^{[\beta]-1}\right), & 1<\beta\le2,\\
		O_2\left(|x|^{2-\min \left\{n,\beta \right\}}(\ln|x|)^{\mu_2}\right), & \beta>2,
	\end{cases}
\end{equation}
as $|x|\to +\infty$. The constant $\mu_2$ is defined as above. 

Moreover, the corresponding local H\"older estimates for the Hessian remainders also hold; namely, if the remainder is denoted by \(\phi\), then
\[
 [D^2\phi]_{C^{0,\alpha}(B_{|x|/2}(x))}
\]
has the decay obtained by differentiating the corresponding \(O_2\) estimate and losing an additional factor \(|x|^{-\alpha}\).
\end{thm}

\begin{rem}
The rates in Theorem~\ref{D2uehd} are also optimal.

First, we consider the critical case \(\beta=1\). Let
\[
u_4(x)=\frac{1}{2}\tan\frac{\theta}{n}|x|^2+g_4(x),
\]
where
\[
g_4(x):=
\begin{cases}
	\text{smooth}, & 0\le |x|\le 2,\\[2mm]
	\displaystyle \left(\sum_{i=1}^n x_i\right)\ln |x|, & |x|>2.
\end{cases}
\]
A direct computation gives
\[
F(D^2u_4(x))
=
\theta+f_4(x)
=
\theta+O_1(|x|^{-1}),
\qquad |x|\to\infty.
\]

For the remaining slow convergence cases, namely
\(\beta\in(0,2]\), we use radial examples. Let \(f_5\) be a smooth
function satisfying
\[
f_5(x):=
\begin{cases}
	\text{smooth}, & 0\le |x|\le 2,\\[2mm]
	|x|^{-\beta}, & |x|>2.
\end{cases}
\]
According to the computations in \cite{LB2023ANS,BLW2024}, the radial solution
of \eqref{eq} with right-hand side \(\theta+f_5\) has the asymptotic behavior
\[
u_5(x)
=
\frac{1}{2}\tan\frac{\theta}{n}|x|^2
+
\begin{cases}
	O_1(|x|^{2-\beta}), & 0<\beta<2,\ \beta\neq1,\\[1mm]
	O_1(\ln |x|), & \beta=2,
\end{cases}
\qquad |x|\to\infty.
\]
These orders agree with the corresponding remainders in Theorem~\ref{D2uehd};
hence they are sharp. Moreover, one can verify that the above \(f_i\) and
\(u_i\), \(i=4,5\), satisfy \eqref{Ducond}.

Finally, for the fast convergence case \(\beta>2\), the optimality of
\eqref{bgn} follows from the global existence result in
\cite{BLW2024}. 
\end{rem}

\begin{rem}
	Note that any function $f$ satisfying \eqref{Ducondo} automatically satisfies \eqref{Ducond}. Hence, Theorems~\ref{D2ue} and~\ref{D2uehd} apply under strictly weaker regularity assumptions on $f$ than Theorem~\ref{D2ueo}.
	On the other hand, in \cite[Theorem~1]{LLY2020}, the right-hand side $f$ is assumed to vanish outside a compact subset of $\mathbb{R}^n$. This clearly implies \eqref{Ducond} for every $\beta>2$. Therefore, the result in \cite[Theorem~1]{LLY2020} can be viewed as a particular case of the asymptotic expansions \eqref{bg2} and \eqref{bgn} under supercritical phase, obtained under a stronger decay assumption on $f$.
	Furthermore, the interior gradient estimate from \cite{AB2021} shows that any solution of \eqref{eq} with a quadratic upper growth bound at infinity must have at most linear growth of its gradient at infinity. In particular, condition~\eqref{Ducond} is automatically fulfilled in this setting. It follows that Theorem~\ref{D2ue} also includes the main result of Liu--Bao \cite{LB2022}.
\end{rem}

\begin{rem}\label{etu}
By an extension argument as in \cite[Theorem~3.2]{Y2014}, we may modify the value of \(f\) on \(\Omega\) without affecting the asymptotic behavior near infinity. Moreover, interior estimates such as those in Lemma~17.16 of \cite{GT}, together with the Hessian estimates of Zhou \cite{Z2025}, imply that the local regularity assumptions used below are available under the hypotheses of the main theorems.
\end{rem}

This paper is organized as follows. In Section~\ref{sec2}, we introduce the Lewy rotation and obtain the Hessian convergence rate at infinity. The main new ingredient is a scale-dependent second difference argument, together with an annular Campanato improvement, which yields a quantitative convergence rate under only Lipschitz regularity of the phase. In Section~\ref{sec3}, we treat the fast convergence case \(\beta>2\) in both dimension two and higher dimensions. In Section~\ref{sec4}, we study the slow convergence case \(0<\beta\le2\).

	\section{Hessian convergence at infinity}\label{sec2}
	\subsection{Lewy rotation in supercritical phase}\label{sec2.1}

	Throughout this paper, we assume that for sufficiently large $|x|$,
	\[
	\theta + f(x) >\frac{(n-2)\pi}{2}+ 2\delta \quad \text{and} \quad |f(x)| \le \delta
	\]
	for some positive constant $\delta$.
	
	Following \cite{LLY2020,Y2002,Yuan2006}, we first perform a transformation of the solution so that the Hessian of the new potential $\tilde{u}$ becomes bounded. Choosing the rotation angle $\vartheta = \delta/n$, define
	\begin{equation}\label{rtfm}
		\tilde{x} = \mathfrak{c}x + \mathfrak{s} Du(x), \qquad 
		\tilde{y} = -\mathfrak{s}x + \mathfrak{c} Du(x),
	\end{equation}
	where $(\mathfrak{c}, \mathfrak{s}) = (\cos\vartheta, \sin\vartheta)$.
	
	Let $\tilde{u}(\tilde{x}) = \int^{\tilde{x}} \tilde{y} \cdot d\tilde{x}$. Then $D_{\tilde{x}}\tilde{u} = \tilde{y}$,
	\[
	D_{\tilde{x}}^2\tilde{u} = \bigl(-\mathfrak{s}I + \mathfrak{c} D^2u\bigr)\bigl(\mathfrak{c}I + \mathfrak{s} D^2u\bigr)^{-1},
	\]
	and $\tilde{u}$ satisfies the equation
	\begin{equation}\label{npeq}
		\sum_{i=1}^n \arctan \lambda_i(D_{\tilde{x}}^2 \tilde{u}) = \tilde{\theta} + \tilde{f}\bigl(\tilde{x}, D_{\tilde{x}}\tilde{u}(\tilde{x})\bigr), \quad 
		|D_{\tilde{x}}^2 \tilde{u}| < \cot\vartheta \quad \text{in } \mathbb{R}^n \setminus \overline{\tilde{\Omega}},
	\end{equation}
	where $\tilde{\theta} = \theta - \delta > 0$, 
	$\tilde{f}\bigl(\tilde{x}, D_{\tilde{x}}\tilde{u}(\tilde{x})\bigr) = f(x) = f\bigl(\mathfrak{c}\tilde{x} - \mathfrak{s} D_{\tilde{x}}\tilde{u}(\tilde{x})\bigr)$, 
	and $\tilde{\Omega} = \tilde{x}(\Omega)$ is a bounded domain.
	
	Consequently, the right-hand side of \eqref{npeq} lies in the supercritical phase, which ensures that $F$ is concave in the level-set sense and can be modified into a concave operator. For further details, see \cite[Lemma~2.2]{CPW2017} and \cite{CW2019}. Without ambiguity, we continue to denote the modified operator by $F$ and regard it as concave.
	
	\begin{rem}
		By Remark~\ref{etu}, together with the fact that the map $x \mapsto \tilde{x}$ is a diffeomorphism from $\mathbb{R}^n$ onto itself (see \cite{LLY2020}), the function $\tilde{u}$ can be defined on $\mathbb{R}^n$ as well. Moreover, $\tilde{u} \in C_{\mathrm{loc}}^{2}(\mathbb{R}^n \setminus \overline{\tilde{\Omega}})$ because the right-hand side of \eqref{npeq} lies in the supercritical phase.
	\end{rem}

	For the new potential in \eqref{npeq}, we have the following proposition.
	\begin{prop}\label{rtp} 
		Let $h(\tilde{x})=\tilde{f} (\tilde{x}, D_{\tilde{x}}\tilde{u} (\tilde{x}))$ and $f_i=\partial f/\partial x_i$, $i=1,\cdots,n$. Suppose that the potential function satisfies \eqref{rtfm} and \eqref{npeq}. Then
		\begin{itemize}
			\item[(i)] 
			$|x|\to +\infty$ as $|\tilde{x}|\to +\infty$.
			\item[(ii)] 
			$|D_{\tilde{x}}h|\leq C(\delta)|D_{x}f|$.
			\item[(iii)]
			Assume that there exists a constant symmetric matrix \(\tilde A\) satisfying
			\(F(\tilde A)=\tilde\theta\) such that
			\[
			D_{\tilde{x}}\tilde{u}(\tilde{x})
			=
			\tilde A\tilde{x}
			+
			O_1(|\tilde{x}|^{\zeta})
			\qquad \text{as } |\tilde{x}|\to+\infty,
			\]
			for some \(\zeta<1\). If $0< \zeta<1$, assume moreover that \(u\) has a quadratic upper growth bound at infinity, namely
			\[
			u(x)\le C(1+|x|^2)
			\qquad \text{as } |x|\to+\infty.
			\]
			Then
			\[
			D_x^2u(x)\to A
			\qquad \text{as } |x|\to+\infty,
			\]
			where
			\[
			A=
			\left(\mathfrak{s}I+\mathfrak{c}\tilde A\right)
			\left(\mathfrak{c}I-\mathfrak{s}\tilde A\right)^{-1}.
			\]
		\end{itemize}
	\end{prop}
\begin{proof}
	Assertion (i) follows immediately from the fact that the map $x\mapsto \tilde x$ is a diffeomorphism from $\mathbb R^n$ onto itself; see \cite[Section 3.1]{LLY2020}.
	
	Assertion (ii) is obtained by a direct computation. Indeed, recalling
	\[
	h(\tilde x)=\tilde f(\tilde x,D_{\tilde x}\tilde u(\tilde x))=f(x),
	\]
	we differentiate $h$ with respect to $\tilde x$ and use the chain rule together with the inverse Lewy rotation to obtain
	\[
	|D_{\tilde x}h(\tilde x)|\le C(\delta)\,|D_xf(x)|.
	\]

We now prove (iii). Assume that there exists a constant symmetric matrix
\(\tilde A\) satisfying \(F(\tilde A)=\tilde\theta\) and
\[
D_{\tilde x}\tilde u(\tilde x)
=
\tilde A\tilde x+O_1(|\tilde x|^\zeta)
\qquad \text{as }|\tilde x|\to+\infty,
\]
for some \(\zeta<1\). If $\zeta\leq 0$, (iii) was proved by Li--Li--Yuan \cite{LLY2020}. Hence, we only consider the case $0<\zeta<1$. By the meaning of the \(O_1\)-notation, we have
\[
D_{\tilde x}^2\tilde u(\tilde x)
=
\tilde A+O(|\tilde x|^{\zeta-1})
\qquad \text{as }|\tilde x|\to+\infty.
\]
In particular,
\[
D_{\tilde x}^2\tilde u(\tilde x)\to \tilde A
\qquad \text{as }|\tilde x|\to+\infty.
\]

We first prove that
\[
\lambda_i(\tilde A)<\cot\vartheta,
\qquad i=1,\ldots,n.
\]
Since the Lewy rotation gives
\[
|D_{\tilde x}^2\tilde u|<\cot\vartheta
\qquad \text{in }\mathbb R^n\setminus\widetilde\Omega,
\]
passing to the limit yields
\[
\lambda_i(\tilde A)\le \cot\vartheta,
\qquad i=1,\ldots,n.
\]
It remains to exclude equality.

Since $u$ has a quadratic upper growth bound at infinity and the Hessian of $u$ is bounded from below, the interior gradient estimate for the supercritical Lagrangian mean
curvature equation \cite{AB2021} yields
\begin{equation}\label{eq:linear_gradient_growth}
	|Du(x)|\le C|x|
	\qquad \text{as } |x|\to+\infty.
\end{equation}

Now suppose, by contradiction, that equality occurs. After an orthogonal rotation in
the \(\tilde x\)-variables, we may assume that \(\tilde A\) is diagonal and
\[
\tilde A_{11}=\cot\vartheta.
\]
Then the asymptotic expansion of \(D_{\tilde x}\tilde u\) gives
\[
\partial_1\tilde u(\tilde x)
=
\cot\vartheta\,\tilde x_1
+
O(|\tilde x|^\zeta)
\qquad \text{as }|\tilde x|\to+\infty.
\]
Using the inverse Lewy rotation formula
\[
x=\mathfrak c\tilde x-\mathfrak sD_{\tilde x}\tilde u(\tilde x),
\]
we obtain
\[
\begin{aligned}
	x_1
	&=
	\mathfrak c\tilde x_1
	-
	\mathfrak s\,\partial_1\tilde u(\tilde x) \\
	&=
	\mathfrak c\tilde x_1
	-
	\mathfrak s
	\left(
	\cot\vartheta\,\tilde x_1
	+
	O(|\tilde x|^\zeta)
	\right).
\end{aligned}
\]
Since
\[
\mathfrak c-\mathfrak s\cot\vartheta=0,
\]
we obtain
\begin{equation}\label{eq:sublinear_strip}
	|x_1|\le C|\tilde x|^\zeta.
\end{equation}

On the other hand, the Lewy rotation is an orthogonal rotation in
\(\mathbb R^n\times\mathbb R^n\). Hence
\[
|\tilde x|^2+|D_{\tilde x}\tilde u(\tilde x)|^2
=
|x|^2+|Du(x)|^2.
\]
Using \eqref{eq:linear_gradient_growth}, we obtain
\[
|\tilde x|
\le
\left(|x|^2+|Du(x)|^2\right)^{1/2}
\le
C(1+|x|).
\]
Together with \eqref{eq:sublinear_strip}, this gives
\[
|x_1|
\le
C(1+|x|)^\zeta
\qquad \text{for } |x|\gg1.
\]
Now take \(x=Re_1\). Since \(\Omega\) is bounded, \(Re_1\in\mathbb R^n\setminus\Omega\)
for all sufficiently large \(R\). Therefore
\[
R=|x_1|
\le
C(1+R)^\zeta.
\]
This is impossible as \(R\to+\infty\), because \(\zeta<1\). Hence equality cannot
occur, and consequently
\[
\lambda_i(\tilde A)<\cot\vartheta,
\qquad i=1,\ldots,n.
\]

It follows that the matrix
\[
\mathfrak c I-\mathfrak s\tilde A
\]
is invertible. Moreover, from
\[
x
=
\mathfrak c\tilde x-\mathfrak sD_{\tilde x}\tilde u(\tilde x)
\]
and the expansion of \(D_{\tilde x}\tilde u\), we obtain
\[
x
=
(\mathfrak cI-\mathfrak s\tilde A)\tilde x
+
O(|\tilde x|^\zeta).
\]
Since \(\mathfrak cI-\mathfrak s\tilde A\) is invertible and \(\zeta<1\), we have
\[
|x|\to+\infty
\qquad \Longleftrightarrow \qquad
|\tilde x|\to+\infty.
\]

Finally, the Hessian transformation formula under the inverse Lewy rotation is
\[
D_x^2u(x)
=
\left(
\mathfrak sI+\mathfrak cD_{\tilde x}^2\tilde u(\tilde x)
\right)
\left(
\mathfrak cI-\mathfrak sD_{\tilde x}^2\tilde u(\tilde x)
\right)^{-1}.
\]
Letting \(|\tilde x|\to+\infty\), and using
\[
D_{\tilde x}^2\tilde u(\tilde x)\to\tilde A,
\]
we conclude that
\[
D_x^2u(x)\to A
\qquad \text{as } |x|\to+\infty,
\]
where
\[
A=
\left(\mathfrak{s}I+\mathfrak{c}\tilde A\right)
\left(\mathfrak{c}I-\mathfrak{s}\tilde A\right)^{-1}.
\]
This proves (iii).
\end{proof}

	In view of the decay condition~\eqref{Ducond} and equation~\eqref{npeq}, we obtain the following lemma.
	
	\begin{lem}\label{abf11}
		Suppose $f$ satisfies \eqref{Ducond}. If the potential function $\tilde{u}$ satisfies \eqref{rtfm} and \eqref{npeq}, then
		\begin{equation}
			\limsup_{|\tilde{x}| \to +\infty} \left(|\tilde{x}|^{\beta} |h(\tilde{x})| + |\tilde{x}|^{\beta+1} [h]_{C^{0,1}(\overline{B_{|\tilde{x}|/2}(\tilde{x})})} \right) < \infty.
		\end{equation}
	\end{lem}
	\begin{proof}
		In view of \eqref{Ducond}, there exist $R_1 \ge R_0$ and a constant $C_1$ such that
		\[
		|x|^{\beta} |f(x)| + |x|^{\beta+1} [f]_{C^{0,1}(\overline{B_{|x|/2}(x)})} \le C_1, \quad \forall |x| \ge R_1.
		\]
		Hence, by \eqref{Ducond}, \eqref{rtfm} and Proposition~\ref{rtp}(i), there exists $R_2$ sufficiently large such that for all $|\tilde{x}| \ge R_2$, we have $|x| \ge R_1$ and
		\[
		|h(\tilde{x})|\,|\tilde{x}|^{\beta} \le C(\beta,\delta)\bigl(|x|^{\beta} + |Du(x)|^{\beta}\bigr) |f(x)| \le C(\delta,\beta,C_0,C_1).
		\]
		
		Similarly, using \eqref{Ducond}, \eqref{rtfm} and Proposition~\ref{rtp}(ii), there exists $R_2$ sufficiently large such that for $|\tilde{x}| \ge R_2$,
		\begin{align*}
			[h]_{C^{0,1}(\overline{B_{|\tilde{x}|/2}(\tilde{x})})} |\tilde{x}|^{\beta+1}
			&\le C(\beta,\delta) \bigl(|x|^{\beta+1} + |Du(x)|^{\beta+1}\bigr) [h]_{C^{0,1}(\overline{B_{|\tilde{x}|/2}(\tilde{x})})} \\
			&= C(\delta,\beta) \bigl(|x|^{\beta+1} + |Du(x)|^{\beta+1}\bigr) \sup_{\tilde{y} \in \overline{B_{|\tilde{x}|/2}(\tilde{x})}} \frac{|h(\tilde{x}) - h(\tilde{y})|}{|\tilde{x} - \tilde{y}|} \\
			&\le C(\delta,\beta) \bigl(|x|^{\beta+1} + |Du(x)|^{\beta+1}\bigr) \sup_{y \in \overline{B_{|x|/2}(x)}} \frac{|f(x) - f(y)|}{|x - y|} \\
			&\le C(\delta,\beta,C_0,C_1),
		\end{align*}
		where the second inequality follows from the fact that the map $x \mapsto \tilde{x}$ is a diffeomorphism and satisfies $|\tilde{x} - \tilde{y}| \ge \sin\vartheta \, |x - y|$; for further details, see \cite{LLY2020}. This completes the proof.
	\end{proof}

\subsection{The scale-dependent Hessian convergence mechanism}\label{subsec:hessian_quantitative}
In this subsection we work after the Lewy rotation and the concavity modification described above. For simplicity, we still denote the rotated variables and the rotated potential by \(x\) and \(u\). 

Let $n\geq 3$ and let $u$ be a viscosity solution of
\begin{equation}\label{eq:main}
    F(D^2u)=\theta+f(x)
    \qquad \text{in } \mathbb{R}^n\setminus\overline{\Omega}.
\end{equation}
We assume that the operator $F$ is concave and uniformly elliptic along $D^2u$. More precisely, assume
\begin{equation}\label{eq:hessian_bound}
    \|D^2u\|_{L^\infty(\R^n\setminus\Omega)}\leq M,
\end{equation}
and that there exist $0<\lambda\leq \Lambda<\infty$ such that
\begin{equation}\label{eq:uniform_ellipticity}
    \lambda I\leq \bigl(F^{ij}(D^2u(x))\bigr)\leq \Lambda I
    \qquad \text{in}\ \mathbb{R}^n\setminus\overline{\Omega}.
\end{equation}
Assume further that \(f\in C^{0,1}_{\mathrm{loc}}\) and satisfies \eqref{Ducond}.

The following result is the higher-dimensional ($n\geq 3$) replacement for the exterior Hessian convergence argument of Li--Li--Yuan~\cite{LLY2020}. The key point is that we use a scale-dependent second difference quotient instead of differentiating the equation twice.

\begin{thm}[Quantitative Hessian convergence]\label{prop:quantitative_hessian}
Under the assumptions above, there exist a matrix $A\in\mathcal A$, a number
$\varepsilon_0>0$, and a constant $C>0$ such that
\[
    D^2u(x)\to A
    \qquad \text{as } |x|\to\infty,
\]
and
\begin{equation}\label{eq:quantitative_convergence}
    |D^2u(x)-A|\leq C|x|^{-\varepsilon_0}
    \qquad \text{for all sufficiently large } |x|.
\end{equation}
More precisely, let $\alpha\in(0,1)$ be the Evans--Krylov exponent. Choose
\[
    0<\sigma<\min\{\beta,1\}
\]
and
\[
    0<\tau<\min\{\sigma\alpha,\beta-\sigma\}.
\]
Then one can take some
\[
    0<\varepsilon_0<
    \min\left\{
    \tau,\,
    -\frac{\log\eta}{\log L}
    \right\},
\]
where $L>2$ and $\eta\in(0,1)$ are the constants in the exterior linear decay lemma
used below.
\end{thm}

\subsection{Scale-dependent second difference quotients}

Fix $e\in\partial B_1$ and write
\[
    \omega(x)=u_{ee}(x).
\]
Let
\[
    0<\sigma<\min\{\beta,1\}.
\]
For a large scale $R$, define
\[
    h_R=R^{1-\sigma}.
\]
Then
\[
    h_R\to\infty,
    \qquad
    \frac{h_R}{R}=R^{-\sigma}\to0.
\]
For $|x|\sim R$, define the second difference quotient
\begin{equation}\label{eq:wR_definition}
    w_R(x)
    =
    \frac{u(x+h_Re)+u(x-h_Re)-2u(x)}{h_R^2}.
\end{equation}

Step 1. Approximation of $u_{ee}$

By the interior Hessian estimate from \cite{Z2025} and the Evans--Krylov estimate applied after rescaling on balls
of radius comparable to $R$, there exists $\alpha\in(0,1)$ such that
\begin{equation}\label{eq:large_scale_holder}
    [D^2u]_{C^\alpha(B_{cR}(z))}
    \leq C R^{-\alpha}
    \qquad \text{whenever } |z|\sim R.
\end{equation}
Since $h_R/R\to0$, for large $R$ we have $B_{h_R}(x)\subset B_{cR}(x)$ whenever $|x|\sim R$.
Using the integral form of the one-dimensional second difference in the direction $e$, we obtain
\begin{equation}\label{eq:wR_approximation}
    |w_R(x)-u_{ee}(x)|
    \leq
    C h_R^\alpha [D^2u]_{C^\alpha(B_{h_R}(x))}
    \leq
    C\left(\frac{h_R}{R}\right)^\alpha
    =
    C R^{-\sigma\alpha}.
\end{equation}

Step 2. Linear inequality for $w_R$

Set
\[
    a^{ij}(x)=F^{ij}(D^2u(x)),
    \qquad
    L\phi=a^{ij}(x)\phi_{ij}.
\]
By the concavity of $F$,
\[
\begin{aligned}
    F(D^2u(x+h_Re))-F(D^2u(x))
    &\leq
    F^{ij}(D^2u(x))
    \bigl(u_{ij}(x+h_Re)-u_{ij}(x)\bigr),\\
    F(D^2u(x-h_Re))-F(D^2u(x))
    &\leq
    F^{ij}(D^2u(x))
    \bigl(u_{ij}(x-h_Re)-u_{ij}(x)\bigr).
\end{aligned}
\]
Adding these two inequalities and dividing by $h_R^2$ gives
\begin{equation}\label{eq:wR_subsolution}
    Lw_R(x)\geq g_R(x),
\end{equation}
where
\begin{equation}\label{eq:gR_definition}
    g_R(x)
    =
    \frac{f(x+h_Re)+f(x-h_Re)-2f(x)}{h_R^2}.
\end{equation}
For $|x|\sim R$ and $R$ large enough, $x\pm h_Re\in B_{|x|/2}(x)$. Therefore, by
\eqref{Ducond},
\[
\begin{aligned}
    |g_R(x)|
    &\leq
    \frac{|f(x+h_Re)-f(x)|+|f(x-h_Re)-f(x)|}{h_R^2} \\
    &\leq
    C h_R^{-1} R^{-\beta-1}
    =
    C R^{-\beta-2+\sigma}.
\end{aligned}
\]
Thus
\begin{equation}\label{eq:gR_decay}
    |g_R(x)|\leq C R^{-\beta-2+\sigma}
    \qquad \text{for } |x|\sim R.
\end{equation}
Equivalently,
\begin{equation}\label{eq:scale_invariant_error}
    R^2\|g_R\|_{L^\infty(\{|x|\sim R\})}
    \leq C R^{\sigma-\beta}.
\end{equation}

Step 3. Local comparison estimate

Let
\[
    D_{\rho_1,\rho_2}=B_{\rho_2}\setminus \overline{B_{\rho_1}},
\]
where $\rho_2/\rho_1$ is bounded above and below by universal constants greater than $1$.
Let $R$ be comparable to $\rho_1$ and $\rho_2$. Suppose that
\[
    \omega\leq K
    \qquad \text{on } \partial D_{\rho_1,\rho_2}.
\]
Then
\begin{equation}\label{eq:local_comparison}
    \omega\leq
    K+C\left(R^{-\sigma\alpha}+R^{\sigma-\beta}\right)
    \qquad \text{in }D_{\rho_1,\rho_2}.
\end{equation}

Indeed, by \eqref{eq:wR_approximation},
\[
    w_R\leq K+C R^{-\sigma\alpha}
    \qquad \text{on } \partial D_{\rho_1,\rho_2}.
\]
Let $\psi_R$ solve
\[
    \begin{cases}
    L\psi_R=-|g_R|,& x\in D_{\rho_1,\rho_2},\\
    \psi_R=0,& x\in \partial D_{\rho_1,\rho_2}.
    \end{cases}
\]
By the maximum principle, $\psi_R\geq0$. Since the annulus has a fixed shape, the Green
function estimate for uniformly elliptic operators gives
\[
    0\leq \psi_R\leq C R^2\|g_R\|_{L^\infty(D_{\rho_1,\rho_2})}.
\]
Using \eqref{eq:gR_decay},
\[
    0\leq \psi_R\leq C R^{\sigma-\beta}.
\]
Now set
\[
    q=w_R-\bigl(K+C R^{-\sigma\alpha}\bigr)-\psi_R.
\]
Then $q\leq0$ on $\partial D_{\rho_1,\rho_2}$, while
\[
    Lq=Lw_R-L\psi_R\geq g_R+|g_R|\geq0.
\]
The maximum principle implies $q\leq0$ in $D_{\rho_1,\rho_2}$. Hence
\[
    w_R\leq K+C\left(R^{-\sigma\alpha}+R^{\sigma-\beta}\right).
\]
Using \eqref{eq:wR_approximation} once more yields \eqref{eq:local_comparison}.

\subsection{Qualitative Hessian convergence in higher dimensions}

We recall the qualitative argument, since it is the input for the quantitative tail Campanato estimate.

\begin{lem}[Hessian convergence]\label{lem:qualitative_hessian}
There exists $A\in\mathcal A$ such that
\[
    D^2u(x)\to A
    \qquad \text{as } |x|\to\infty.
\]
\end{lem}

\begin{proof}
Fix $e\in\partial B_1$ and set $\omega=u_{ee}$. Define
\[
    \overline\omega=\limsup_{|x|\to\infty}\omega(x),
    \qquad
    \underline\omega=\liminf_{|x|\to\infty}\omega(x).
\]
We show that $\overline\omega=\underline\omega$.

Assume by contradiction that
\[
    \overline\omega-\underline\omega=5d>0.
\]
Choose $0<\varepsilon<d$. There exists $R_\varepsilon>0$ such that
\[
    \omega(x)\leq \overline\omega+\varepsilon
    \qquad \text{for } |x|\geq R_\varepsilon.
\]
There also exists a sequence $x_k\to\infty$, $r_k=|x_k|\to\infty$, such that
\[
    \omega(x_k)\leq \underline\omega+\varepsilon.
\]

We first claim that for some sufficiently large $k$ there is a point
$y_k\in\partial B_{r_k}$ such that
\[
    \omega(y_k)\geq \overline\omega-\varepsilon.
\]
If not, then after passing to a tail,
\[
    \sup_{\partial B_{r_k}}\omega<\overline\omega-\varepsilon
    \qquad \text{for all large }k.
\]
Take two consecutive radii $r_k<r_{k+1}$ and connect them by a finite chain of radii
\[
    r_k=\rho_0<\rho_1<\cdots<\rho_N=r_{k+1},
\]
with all ratios $\rho_{j+1}/\rho_j$ bounded above and below by universal constants greater than
$1$. Applying the local comparison estimate \eqref{eq:local_comparison} on the annuli
\[
    B_{\rho_{j+1}}\setminus \overline{B_{\rho_{j-1}}},
    \qquad 1\leq j\leq N-1,
\]
and using the discrete maximum principle, one obtains
\[
    \sup_{\partial B_{\rho_j}}\omega
    \leq
    \overline\omega-\varepsilon
    +
    C\sum_{m=0}^{N}
    \left(\rho_m^{-\sigma\alpha}+\rho_m^{\sigma-\beta}\right).
\]
Since the radii grow geometrically and $\sigma\alpha>0$, $\beta-\sigma>0$, the sum is bounded by
\[
    C\left(r_k^{-\sigma\alpha}+r_k^{\sigma-\beta}\right).
\]
For $k$ large, this is smaller than $\varepsilon/4$. Applying the local comparison estimate once
again on each annulus $B_{\rho_{j+1}}\setminus \overline{B_{\rho_j}}$, we obtain
\[
    \omega(x)\leq \overline\omega-\frac{\varepsilon}{2}
    \qquad \text{for } r_k\leq |x|\leq r_{k+1}.
\]
Since such intervals cover the exterior region for large $k$, this contradicts the definition of
$\overline\omega$. The claim follows.

Choose \(k\) sufficiently large and set
\[
x^-:=x_k,\qquad x^+:=y_k,\qquad r:=r_k.
\]
Then
\[
|x^-|=|x^+|=r,
\]
and
\[
\omega(x^-)\le \underline{\omega}+\varepsilon,
\qquad
\omega(x^+)\ge \overline{\omega}-\varepsilon.
\]

 By the interior
\(C^{2,\alpha}\) estimate in \eqref{eq:large_scale_holder}, there exist constants \(C>0\) and
\(\alpha\in(0,1)\), independent of \(r\), such that for every sufficiently large
\(r\) and every \(z\in\partial B_r\),
\[
[D^2u]_{C^\alpha(B_{\gamma r}(z))}
\le C r^{-\alpha}.
\]
Consequently, 
\[
\osc_{B_{\gamma r}(z)}\omega
\le
C(\gamma r)^\alpha r^{-\alpha}
=
C\gamma^\alpha.
\]
We now fix \(\gamma\in(0,1/10)\) sufficiently small so that
\[
C\gamma^\alpha\le d.
\]
 Then
\[
\osc_{B_{\gamma r}(z)}\omega\le d
\qquad \text{for every } |z|=r.
\]

In particular, applying this with \(z=x^-\), we obtain
\[
\omega(x)
\le
\omega(x^-)+d
\le
\underline{\omega}+\varepsilon+d
=
\overline{\omega}-4d+\varepsilon
\le
\overline{\omega}-3d
\qquad
\text{for } x\in B_{\gamma r}(x^-),
\]
where we used
\[
\overline{\omega}-\underline{\omega}=5d
\qquad\text{and}\qquad
0<\varepsilon<d.
\]

Now set $h_r=r^{1-\sigma}$ and define $w_r$ by \eqref{eq:wR_definition}. By
\eqref{eq:wR_approximation},
\[
    |w_r-\omega|\leq C r^{-\sigma\alpha}
    \qquad \text{for } |x|\sim r.
\]
For $r$ large, this implies
\[
    w_r\leq \overline\omega-\frac{5d}{2}
    \qquad \text{in }B_{\gamma r}(x^-),
\]
while
\[
    w_r(x^+)\geq \overline\omega-\varepsilon-Cr^{-\sigma\alpha}.
\]
Define
\[
    v(x)=\overline\omega+\varepsilon+Cr^{-\sigma\alpha}-w_r(x),
\]
where $C$ is chosen so that $v\geq0$ in the fixed-shape annulus
\[
    A_r=B_{(1+3\gamma)r}\setminus \overline{B_{(1-3\gamma)r}}.
\]
By \eqref{eq:wR_subsolution},
\[
    Lv=-Lw_r\leq -g_r\leq |g_r|.
\]
Moreover, $v\geq 2d$ in $B_{\gamma r}(x^-)$ and
\[
    v(x^+)\leq 2\varepsilon+Cr^{-\sigma\alpha}.
\]
Applying the weak Harnack inequality in $A_r$ yields, for some $p>0$,
\[
    \left(
    \frac{1}{|B_{\gamma r}(x^-)|}
    \int_{B_{\gamma r}(x^-)}v^p
    \right)^{1/p}
    \leq
    C\left(
    \inf_{A_r}v
    +
    r\|g_r\|_{L^n(A_r)}
    \right).
\]
By \eqref{eq:gR_decay},
\[
    r\|g_r\|_{L^n(A_r)}
    \leq
    C r\cdot r^{-\beta-2+\sigma}\cdot |A_r|^{1/n}
    \leq
    C r^{\sigma-\beta}.
\]
Thus
\[
    2d
    \leq
    C\left(\varepsilon+r^{-\sigma\alpha}+r^{\sigma-\beta}\right).
\]
Letting $r\to\infty$ and then $\varepsilon\to0$ gives $d=0$, a contradiction.
Therefore, $\overline\omega=\underline\omega$ and there exists a symmetric matrix $A$ such that
\[
    D^2u(x)\to A.
\]
Passing to the limit in \eqref{eq:main}, using $f(x)\to0$, gives $F(A)=\theta$.
Therefore $A\in\mathcal A$.
\end{proof}

\subsection{Annular Campanato improvement}

The qualitative convergence now allows us to define a tail excess relative to the level set
$\mathcal A$.

Let
\[
    \Sigma=\mathcal A=\{B\in S^{n\times n}:F(B)=\theta\}.
\]
For $R$ sufficiently large, define
\begin{equation}\label{eq:tail_excess}
    \Phi(R)
    =
    \inf_{B\in\Sigma}
    \sup_{|x|\geq R}|D^2u(x)-B|.
\end{equation}
By Lemma~\ref{lem:qualitative_hessian}, $\Phi(R)\to0$ as $R\to\infty$.

We shall use the following standard exterior decay lemma for constant-coefficient linear equations.

\begin{lem}[Annular exterior linear decay]\label{lem:linear_decay}
	Let
	\[
	L_A v=F^{ij}(A)v_{ij},
	\]
	and let
	\[
	T_A\Sigma
	=
	\{P\in S^{n\times n}:F^{ij}(A)P_{ij}=0\}
	\]
	be the tangent space of $\Sigma$ at $A$. There exist $L>2$ and
	$\eta_*\in(0,1)$, depending only on $n,\lambda,\Lambda$, such that if
	\[
	L_Av=0
	\qquad \text{in } \mathbb R^n\setminus B_1,
	\]
	and $D^2v$ is bounded in $\mathbb R^n\setminus B_1$, then
	\begin{equation}\label{eq:annular_linear_decay}
		\inf_{P\in T_A\Sigma}
		\sup_{L\le |y|\le L^2}|D^2v(y)-P|
		\le
		\eta_*
		\inf_{P\in T_A\Sigma}
		\sup_{1\le |y|\le L}|D^2v(y)-P|.
	\end{equation}
\end{lem}

\begin{proof}
	Since the matrix $\bigl(F^{ij}(A)\bigr)$ is positive definite, after a fixed
	linear change of variables the operator
	\[
	L_A=F^{ij}(A)\partial_{ij}
	\]
	is transformed into the Laplacian. This change of variables only changes the
	constants in the estimate. Hence it is enough to prove the estimate for exterior
	harmonic functions. For simplicity of notation, we prove it for
	\[
	\Delta v=0
	\qquad \text{in } \mathbb R^n\setminus B_1.
	\]
	
	Assume that $D^2v$ is bounded in $\mathbb R^n\setminus B_1$. By the exterior
	spherical harmonic expansion, together with the boundedness of $D^2v$, we have
	\[
	v(y)=q_2(y)+q_1(y)+q_0+\phi(y),
	\]
	where $q_2$ is a harmonic quadratic polynomial, $q_1$ is linear, $q_0$ is
	constant, and $\phi$ is the decaying exterior harmonic part. Consequently,
	\[
	D^2v(y)=D^2q_2+D^2\phi(y).
	\]
	Set
	\[
	P_\infty:=D^2q_2.
	\]
	Then
	\[
	D^2v(y)-P_\infty=D^2\phi(y).
	\]
	
	The slowest decaying exterior harmonic mode is the fundamental solution
	$|y|^{2-n}$, whose Hessian decays like $|y|^{-n}$. All higher exterior modes
	decay faster. Hence there exists a constant $C$, depending only on $n$, such
	that for every $L\ge2$,
	\begin{equation}\label{eq:decay_from_inner_annulus}
		\sup_{L\le |y|\le L^2}|D^2v(y)-P_\infty|
		\le
		C L^{-n}
		\sup_{1\le |y|\le L}|D^2v(y)-P_\infty|.
	\end{equation}
	
	We next show that $P_\infty\in T_A\Sigma$ in the original variables. Indeed,
	before the linear change of variables, $v$ satisfies
	\[
	L_Av=F^{ij}(A)v_{ij}=0.
	\]
	Since $D^2v(y)\to P_\infty$ as $|y|\to\infty$, passing to the limit in this
	equation gives
	\[
	F^{ij}(A)(P_\infty)_{ij}=0.
	\]
	Therefore
	\[
	P_\infty\in T_A\Sigma.
	\]
	
	Now define
	\[
	E_1
	:=
	\inf_{P\in T_A\Sigma}
	\sup_{1\le |y|\le L}|D^2v(y)-P|.
	\]
	We claim that
	\begin{equation}\label{eq:pinfty_control_by_E1_annulus}
		\sup_{1\le |y|\le L}|D^2v(y)-P_\infty|
		\le
		2 E_1,
	\end{equation}
	where $C$ is independent of $v$ and $L$.
	
%
	Let \(\delta>0\). Choose \(P_\delta\in T_A\Sigma\) such that
	\[
	\sup_{1\le |y|\le L}|D^2v(y)-P_\delta|
	\le E_1+\delta.
	\]
	We claim that
	\[
	|P_\infty-P_\delta|
	\le E_1+\delta.
	\]
	Indeed, after the linear change of variables which reduces \(L_A\) to the
	Laplacian, we may write
	\[
	v(y)=q_2(y)+q_1(y)+q_0+\phi(y),
	\]
	where \(q_2\) is a harmonic quadratic polynomial, \(q_1\) is linear, \(q_0\) is
	constant, and \(\phi\) is the decaying exterior harmonic part. Hence
	\[
	D^2v(y)=P_\infty+D^2\phi(y),
	\qquad
	P_\infty=D^2q_2.
	\]
	For every \(r\ge1\), the spherical average of \(D^2\phi\) vanishes:
	\[
	\frac{1}{|\partial B_r|}\int_{\partial B_r}D^2\phi(y)\,dS_y=0.
	\]
	Consequently,
	\[
	\begin{aligned}
		P_\infty-P_\delta
		&=
		\frac{1}{|\partial B_r|}\int_{\partial B_r}
		\bigl(P_\infty+D^2\phi(y)-P_\delta\bigr)\,dS_y \\
		&=
		\frac{1}{|\partial B_r|}\int_{\partial B_r}
		\bigl(D^2v(y)-P_\delta\bigr)\,dS_y.
	\end{aligned}
	\]
	Therefore, for any \(1\le r\le L\),
	\[
	|P_\infty-P_\delta|
	\le
	\frac{1}{|\partial B_r|}\int_{\partial B_r}|D^2v(y)-P_\delta|\,dS_y
	\le
	\sup_{1\le |y|\le L}|D^2v(y)-P_\delta|
	\le
	E_1+\delta.
	\]
	It follows that
	\[
	\begin{aligned}
		\sup_{1\le |y|\le L}|D^2v(y)-P_\infty|
		&\le
		\sup_{1\le |y|\le L}|D^2v(y)-P_\delta|
		+
		|P_\delta-P_\infty| \\
		&\le
		2E_1+2\delta.
	\end{aligned}
	\]
	Letting \(\delta\to0\), we obtain
	\[
	\sup_{1\le |y|\le L}|D^2v(y)-P_\infty|
	\le 2E_1.
	\]

	Combining \eqref{eq:decay_from_inner_annulus} and
	\eqref{eq:pinfty_control_by_E1_annulus}, we obtain
	\[
	\sup_{L\le |y|\le L^2}|D^2v(y)-P_\infty|
	\le
	C L^{-n} E_1.
	\]
	Since $P_\infty\in T_A\Sigma$, it follows that
	\[
	\begin{aligned}
		\inf_{P\in T_A\Sigma}
		\sup_{L\le |y|\le L^2}|D^2v(y)-P|
		&\le
		\sup_{L\le |y|\le L^2}|D^2v(y)-P_\infty| \\
		&\le
		C L^{-n}
		\inf_{P\in T_A\Sigma}
		\sup_{1\le |y|\le L}|D^2v(y)-P|.
	\end{aligned}
	\]
	Finally, choose $L>2$ sufficiently large so that
	\[
	C L^{-n}\le \eta_*<1.
	\]
	Then \eqref{eq:annular_linear_decay} follows.
\end{proof}

\begin{lem}[One-step annular improvement]\label{lem:one_step_annular_improvement}
	Fix
	\[
	0<\sigma<\min\{\beta,1\}
	\]
	and choose
	\[
	0<\tau<\min\{\sigma\alpha,\beta-\sigma\}.
	\]
	Let $L>2$ and $\eta_*\in(0,1)$ be as in Lemma~\ref{lem:linear_decay}, and choose
	\[
	\eta_*<\eta<1.
	\]
	For $R$ sufficiently large, define the annular excess
	\[
	\Psi(R)
	:=
	\inf_{B\in\Sigma}
	\sup_{R\le |x|\le LR}|D^2u(x)-B|.
	\]
	Then there exist $R_0>0$ and $C>0$ such that for all $R\ge R_0$,
	\begin{equation}\label{eq:one_step_annular}
		\Psi(LR)\le \eta\,\Psi(R)+C R^{-\tau}.
	\end{equation}
\end{lem}

\begin{proof}
	We argue by contradiction. Suppose that \eqref{eq:one_step_annular} fails.
	Then there exist $R_k\to\infty$ such that
	\begin{equation}\label{eq:annular_failure}
		\Psi(LR_k)>
		\eta\Psi(R_k)+kR_k^{-\tau}.
	\end{equation}
	Set
	\[
	\delta_k:=\Psi(R_k).
	\]
	Since $D^2u(x)\to A$ at infinity, we have
	\[
	\delta_k\to0.
	\]
	Moreover, since
	\[
	\Psi(LR_k)\le \Psi(R_k)=\delta_k,
	\]
	the failure assumption gives
	\[
	(1-\eta)\delta_k>kR_k^{-\tau}.
	\]
	Hence
	\begin{equation}\label{eq:annular_small_tau}
		\frac{R_k^{-\tau}}{\delta_k}\to0.
	\end{equation}
	Since $\tau<\beta$, we also have
	\begin{equation}\label{eq:annular_small_beta}
		\frac{R_k^{-\beta}}{\delta_k}\to0.
	\end{equation}
	
	Choose $B_k\in\Sigma$ such that
	\begin{equation}\label{eq:annular_Bk_choice}
		\sup_{R_k\le |x|\le LR_k}|D^2u(x)-B_k|
		\le (1+o(1))\delta_k.
	\end{equation}
	Since $D^2u(x)\to A$ and $B_k\in\Sigma$, we have
	\[
	B_k\to A.
	\]
	
	Define
	\[
	v_k(y)
	=
	\frac{
		u(R_ky)-\frac12R_k^2y^TB_ky-\ell_k\cdot y-c_k
	}{R_k^2\delta_k}.
	\]
	The affine constants $\ell_k$ and $c_k$ are irrelevant for the Hessian; they may
	be chosen for normalization. We have
	\begin{equation}\label{eq:annular_D2vk}
		D^2v_k(y)
		=
		\frac{D^2u(R_ky)-B_k}{\delta_k}.
	\end{equation}
	By \eqref{eq:annular_Bk_choice},
	\begin{equation}\label{eq:annular_D2vk_bound}
		\sup_{1\le |y|\le L}|D^2v_k(y)|
		\le 1+o(1).
	\end{equation}
	
	The equation for $v_k$ is obtained from
	\[
	F(B_k+\delta_kD^2v_k(y))-F(B_k)=f(R_ky),
	\]
	since $F(B_k)=\theta$. Dividing by $\delta_k$ gives
	\begin{equation}\label{eq:annular_vk_equation}
		a_k^{ij}(y)(v_k)_{ij}
		=
		\frac{f(R_ky)}{\delta_k},
	\end{equation}
	where
	\[
	a_k^{ij}(y)
	=
	\int_0^1
	F^{ij}\bigl(B_k+t\delta_kD^2v_k(y)\bigr)\,dt.
	\]
	Since $B_k\to A$, $\delta_k\to0$, and $D^2v_k$ is uniformly bounded on every
	fixed compact annulus, we have
	\[
	a_k^{ij}\to F^{ij}(A)
	\]
	locally uniformly in $\mathbb R^n\setminus B_1$.
	
	The right-hand side in \eqref{eq:annular_vk_equation} tends to zero locally
	uniformly. Indeed, by the decay of $f$ and \eqref{eq:annular_small_beta}, for
	every fixed compact set $K\subset \mathbb R^n\setminus B_1$,
	\[
	\left\|
	\frac{f(R_k\cdot)}{\delta_k}
	\right\|_{L^\infty(K)}
	\le
	C_K\frac{R_k^{-\beta}}{\delta_k}
	\to0.
	\]
	The Lipschitz decay assumption on $f$ similarly gives local $C^\alpha$ smallness
	for every $\alpha\in(0,1)$.
	
	By Evans--Krylov and Schauder estimates, after passing to a subsequence,
	\[
	v_k\to v_\infty
	\]
	in $C^{2,\alpha'}_{\mathrm{loc}}(\mathbb R^n\setminus B_1)$ for some
	$\alpha'\in(0,\alpha)$. The limit satisfies
	\begin{equation}\label{eq:annular_limit_equation}
		F^{ij}(A)(v_\infty)_{ij}=0
		\qquad \text{in } \mathbb R^n\setminus B_1.
	\end{equation}
	
	From \eqref{eq:annular_D2vk_bound} and the locally uniform convergence of
	$D^2v_k$, we obtain
	\[
	|D^2v_\infty(y)|\le1
	\qquad \text{for every }1\le |y|\le L.
	\]
	Since $0\in T_A\Sigma$,
	\begin{equation}\label{eq:annular_limit_excess_inner}
		\inf_{P\in T_A\Sigma}
		\sup_{1\le |y|\le L}|D^2v_\infty(y)-P|
		\le1.
	\end{equation}
	
	We now pass the failure assumption to the limit on the outer annulus. Define
	\[
	\Sigma_k:=\frac{\Sigma-B_k}{\delta_k}.
	\]
	Since $\Sigma$ is a smooth hypersurface near $A$, $B_k\to A$, and
	$\delta_k\to0$, the sets $\Sigma_k$ converge locally in the Hausdorff sense to
	the tangent space
	\[
	T_A\Sigma
	=
	\{P\in S^{n\times n}:F^{ij}(A)P_{ij}=0\}.
	\]
	Moreover,
	\[
	\begin{aligned}
		\frac{\Psi(LR_k)}{\delta_k}
		&=
		\inf_{B\in\Sigma}
		\sup_{LR_k\le |x|\le L^2R_k}
		\frac{|D^2u(x)-B|}{\delta_k} \\
		&=
		\inf_{Q\in\Sigma_k}
		\sup_{L\le |y|\le L^2}
		|D^2v_k(y)-Q|.
	\end{aligned}
	\]
	By \eqref{eq:annular_failure} and \eqref{eq:annular_small_tau},
	\[
	\frac{\Psi(LR_k)}{\delta_k}
	>
	\eta+\frac{kR_k^{-\tau}}{\delta_k}
	\ge \eta+o(1).
	\]
	Hence
	\begin{equation}\label{eq:annular_scaled_failure}
		\inf_{Q\in\Sigma_k}
		\sup_{L\le |y|\le L^2}|D^2v_k(y)-Q|
		\ge \eta+o(1).
	\end{equation}
	Since the annulus $\{L\le |y|\le L^2\}$ is compact, the convergence
	$D^2v_k\to D^2v_\infty$ is uniform there. Together with the local Hausdorff
	convergence $\Sigma_k\to T_A\Sigma$, we can pass to the limit in
	\eqref{eq:annular_scaled_failure} and obtain
	\begin{equation}\label{eq:annular_limit_excess_outer}
		\inf_{P\in T_A\Sigma}
		\sup_{L\le |y|\le L^2}|D^2v_\infty(y)-P|
		\ge \eta.
	\end{equation}
	
	On the other hand, applying Lemma~\ref{lem:linear_decay} to $v_\infty$ and using
	\eqref{eq:annular_limit_excess_inner}, we obtain
	\[
	\begin{aligned}
		\inf_{P\in T_A\Sigma}
		\sup_{L\le |y|\le L^2}|D^2v_\infty(y)-P|
		&\le
		\eta_*
		\inf_{P\in T_A\Sigma}
		\sup_{1\le |y|\le L}|D^2v_\infty(y)-P| \\
		&\le
		\eta_*.
	\end{aligned}
	\]
	Since $\eta_*<\eta$, this contradicts
	\eqref{eq:annular_limit_excess_outer}. Therefore the one-step improvement
	\eqref{eq:one_step_annular} holds.
\end{proof}

\subsection{Iteration and convergence to the fixed limit}

We now finish the proof of Theorem~\ref{prop:quantitative_hessian}.

\begin{proof}[Proof of Theorem~\ref{prop:quantitative_hessian}]
By Lemma~\ref{lem:qualitative_hessian}, there exists $A\in\Sigma$ such that
\[
    D^2u(x)\to A.
\]
Let $\Phi$ be defined by \eqref{eq:tail_excess}. By Lemma~\ref{lem:one_step_annular_improvement},
\[
    \Phi(LR)\leq \eta\Phi(R)+C R^{-\tau}
\]
for all $R$ large.

Let
\[
R_j=L^jR_0
\]
for some fixed $R_0$.

From the one-step estimate,
\[
\Phi(R_{j+1})\le \eta\Phi(R_j)+CR_j^{-\tau}.
\]
Iterating this inequality gives
\[
\Phi(R_j)
\le
\eta^j\Phi(R_0)
+
C\sum_{m=0}^{j-1}\eta^{j-1-m}R_m^{-\tau}.
\]
Set
\[
\kappa=-\frac{\log\eta}{\log L}>0.
\]
Then
\[
\eta=L^{-\kappa},
\qquad
\eta^j=L^{-j\kappa}
=
R_0^\kappa R_j^{-\kappa}.
\]
Hence
\[
\eta^j\Phi(R_0)=R_0^\kappa\Phi(R_0)R_j^{-\kappa}\le C(R_0) R_j^{-\kappa}.
\]

It remains to estimate the convolution term. Since \(R_m=L^mR_0\), we have
\[
\begin{aligned}
	\sum_{m=0}^{j-1}\eta^{j-1-m}R_m^{-\tau}
	&=
	R_0^{-\tau}
	\sum_{m=0}^{j-1}
	L^{-\kappa(j-1-m)}L^{-m\tau} \\
	&=
	R_0^{-\tau}
	L^{-\kappa(j-1)}
	\sum_{m=0}^{j-1}
	L^{m(\kappa-\tau)}.
\end{aligned}
\]
If \(\kappa<\tau\), the last sum is uniformly bounded, and hence the convolution
term is bounded by \(CR_j^{-\kappa}\). If \(\kappa>\tau\), the last sum is bounded
by \(CL^{j(\kappa-\tau)}\), and hence the convolution term is bounded by
\(CR_j^{-\tau}\). If \(\kappa=\tau\), the last sum equals \(j\), and for every
\(\varepsilon_0<\kappa\),
\[
jL^{-j\kappa}\le C_{\varepsilon_0}L^{-j\varepsilon_0}.
\]
Therefore, in all cases, for every
\[
0<\varepsilon_0<\min\{\tau,\kappa\},
\]
we have
\[
\sum_{m=0}^{j-1}\eta^{j-1-m}R_m^{-\tau}
\le
C R_j^{-\varepsilon_0}.
\]
Combining this with the estimate of the first term gives
\[
\Phi(R_j)\le C R_j^{-\varepsilon_0}.
\]

It remains to convert the decay of $\Phi$ into decay relative to the fixed limiting matrix $A$.
For each large $R$, choose $B_R\in\Sigma$ such that
\[
    \sup_{|x|\geq R}|D^2u(x)-B_R|\leq 2\Phi(R).
\]
Letting $|x|\to\infty$ and using $D^2u(x)\to A$, we obtain
\[
    |A-B_R|\leq 2\Phi(R).
\]
Therefore
\[
\begin{aligned}
    \sup_{|x|\geq R}|D^2u(x)-A|
    &\leq
    \sup_{|x|\geq R}|D^2u(x)-B_R|+|B_R-A|\\
    &\leq
    4\Phi(R)
    \leq
    C R^{-\varepsilon_0}.
\end{aligned}
\]
Taking $R=|x|/2$ gives
\[
    |D^2u(x)-A|\leq C|x|^{-\varepsilon_0}
\]
for all sufficiently large $|x|$. This proves \eqref{eq:quantitative_convergence}.
\end{proof}

\begin{rem}
The role of the scale-dependent difference quotient is exactly to recover the same type of forcing
smallness that one would obtain from $f_{ee}$ in the $C^2$ setting. In the present Lipschitz setting,
\[
    |g_R|\leq C R^{-\beta-2+\sigma},
\]
so the scale-invariant forcing error is
\[
    R^2\|g_R\|_{L^\infty}\leq C R^{\sigma-\beta}.
\]
The approximation error between $w_R$ and $u_{ee}$ is
\[
    |w_R-u_{ee}|\leq C R^{-\sigma\alpha}.
\]
Thus the natural error exponent is
\[
    \min\{\sigma\alpha,\beta-\sigma\}.
\]
This is why, for any
\[
    0<\tau<\min\{\sigma\alpha,\beta-\sigma\},
\]
the tail Campanato iteration yields a polynomial convergence rate.
\end{rem}

In the next two sections, we shall use several standard facts about the fractional Laplacian; we refer the reader to \cite{S2007,CLM2020,PP2015,B2016,QB2025,QB20252n,AH1996,DPV2012} for details.

\section{\texorpdfstring{The fast convergence case $\beta>2$}{The fast convergence case beta > 2}}\label{sec3}
	We first recall the following a priori estimate from \cite{BJ2026}.
	
	Consider the linear elliptic equation
	\begin{equation}\label{leq}
		L(u) = a_{11}(x) u_{11}(x) + 2a_{12}(x) u_{12}(x) + a_{22}(x) u_{22}(x) = f(x),
	\end{equation}
	where $L$ is uniformly elliptic.
	
	\begin{thm}[{\cite[Theorem~2.8]{BJ2026}}]\label{cadu}
		Let $\Omega$ be a bounded domain in $\mathbb{R}^2$ and suppose $f(x) \le C|x|^{-\alpha_0}$ for some $\alpha_0 > 1$ and for all $x \in \mathbb{R}^2 \setminus \overline{\Omega}$. Assume $v \in C^2(\mathbb{R}^2 \setminus \overline{\Omega})$ is a solution of \eqref{leq} in $\mathbb{R}^2 \setminus \overline{\Omega}$ satisfying $|Dv| \le M$. Then for any $\Omega'\supset\Omega$ with $d := \operatorname{dist}(\Omega, \partial\Omega')$, the gradient $Dv(x)$ admits a limit $Dv(\infty)$ at infinity, and moreover
		\begin{equation}\label{abdu}
			|Dv(x) - Dv(\infty)| \le C |x|^{-\alpha}, \qquad x \in \mathbb{R}^2 \setminus \overline{\Omega'},
		\end{equation}
		where $\alpha$ depends only on $\lambda$, $\Lambda$, and $\alpha_0$, while $C$ depends only on $\lambda$, $\Lambda$, $d$, $\alpha_0$, and $M$.
	\end{thm}

	We shall use the theorem above to prove the following result.
\begin{thm}\label{wdecay}
	Let $u$ be a $C_{\mathrm{loc}}^2$ solution of 
	\begin{equation}\label{thm2.10eq}
		G(D^2u)=f(x),\quad x\in \mathbb{R}^2\setminus\overline{\Omega},
	\end{equation}
	where $G\in C^2$ is a fully nonlinear uniformly elliptic operator with ellipticity constants $\lambda$ and $\Lambda$, and $\Omega$ is a bounded domain in $\mathbb{R}^2$.
	Suppose that $\|D^2u\|_{L^{\infty}(\mathbb{R}^2\setminus\overline{\Omega})}\leq M$ and that $f$ satisfies \eqref{Ducond} with $\beta>0$. Then there exists a matrix $A$ such that $G(A)=0$ and
	\[
D^2u=A+O(|x|^{-\varepsilon})\qquad \text{as } |x|\to\infty,
\]
	which implies
	\begin{equation}\label{gru}
		\left|u(x) - \frac{1}{2}x^\mathsf{T} Ax\right| \leq C|x|^{2-\varepsilon}, \qquad |x|\geq R_0,
	\end{equation}
	where $\varepsilon \in (0, 1)$ and $C$ are positive constants depending only on $\lambda$, $\Lambda$, $\beta$, and $M$, and $R_0$ is sufficiently large.
\end{thm}
\begin{proof}
	Fix \(k=1,2\) and choose \(h_0>0\) sufficiently small. For \(0<h<h_0\), define the
	difference quotient
	\[
	v_h(x)=\frac{u(x+he_k)-u(x)}{h},
	\]
	where \(e_k\) denotes the \(k\)-th coordinate vector in \(\mathbb R^2\).
	
	We first fix a bounded domain \(\Omega_*\) such that, for every \(0<h<h_0\),
	\[
	\Omega\cup(\Omega-he_k)\subset \Omega_*,
	\]
	and such that whenever \(x\in \mathbb R^2\setminus\Omega_*\), the segment
	\[
	\{x+the_k:0\le t\le h\}
	\]
	lies in \(\mathbb R^2\setminus\Omega\). Thus \(v_h\) is well defined and smooth in
	\(\mathbb R^2\setminus\overline{\Omega_*}\).
	
	Taking the difference of equation~\eqref{thm2.10eq}, we obtain
	\[
	\frac{
		G(D^2u(x+he_k))-G(D^2u(x))
	}{h}
	=
	\frac{f(x+he_k)-f(x)}{h}.
	\]
	By the mean value formula in the matrix variable, \(v_h\) satisfies
	\[
	a_{ij}^h(x)(v_h)_{ij}(x)
	=
	\frac{f(x+he_k)-f(x)}{h},
	\qquad
	x\in\mathbb R^2\setminus\overline{\Omega_*},
	\]
	where
	\[
	a_{ij}^h(x)
	=
	\int_0^1
	G_{M_{ij}}
	\bigl(
	(1-s)D^2u(x)+sD^2u(x+he_k)
	\bigr)\,ds.
	\]
	Since \(G\) is uniformly elliptic, the coefficients \(a_{ij}^h\) are uniformly
	elliptic with the same ellipticity constants \(\lambda,\Lambda\), independently
	of \(h\).
	
	Moreover, since
	\[
	\|D^2u\|_{L^\infty(\mathbb R^2\setminus\overline{\Omega})}\le M,
	\]
	we have
	\[
	|Dv_h(x)|
	=
	\left|
	\frac{Du(x+he_k)-Du(x)}{h}
	\right|
	\le M,
	\qquad
	x\in\mathbb R^2\setminus\overline{\Omega_*}.
	\]
	This bound is also independent of \(h\).
	
	Next, by the decay assumption on \(f\), for sufficiently large \(|x|\),
	\[
	[f]_{C^{0,1}(B_{|x|/2}(x))}
	\le C|x|^{-\beta-1}.
	\]
	For such \(x\) and \(0<h<h_0\), we have \(x+he_k\in B_{|x|/2}(x)\). Hence
	\[
	\left|
	\frac{f(x+he_k)-f(x)}{h}
	\right|
	\le
	[f]_{C^{0,1}(B_{|x|/2}(x))}
	\le
	C|x|^{-\beta-1}.
	\]
	Since \(\beta+1>1\), the right-hand side satisfies the decay condition required in
	Theorem~\ref{cadu}.
	
	Applying Theorem~\ref{cadu} to \(v_h\) in
	\(\mathbb R^2\setminus\overline{\Omega_*}\), we obtain a vector
	\[
	p_{k,h}:=Dv_h(\infty)\in\mathbb R^2
	\]
	and constants \(C>0\), \(\varepsilon\in(0,1)\), and \(R_1>0\), independent of
	\(h\), such that
	\[
	|Dv_h(x)-p_{k,h}|
	\le C|x|^{-\varepsilon},
	\qquad |x|\ge R_1.
	\]
	Here \(\varepsilon\) depends only on \(\lambda,\Lambda\), and \(\beta\), while
	\(C\) is independent of \(h\).
	
	We now pass to the limit as \(h\to0\). Since \(|Dv_h|\le M\), we have
	\[
	|p_{k,h}|\le M.
	\]
	Thus, for every sequence \(h_j\to0\), there exists a subsequence, still denoted by
	\(h_j\), such that
	\[
	p_{k,h_j}\to p_k
	\]
	for some \(p_k\in\mathbb R^2\).
	
	For every fixed \(x\in\mathbb R^2\setminus\overline{\Omega_*}\), since \(u\in C_{\mathrm{loc}}^2\),
	\[
	Dv_{h_j}(x)
	=
	\frac{Du(x+h_je_k)-Du(x)}{h_j}
	\to
	Du_k(x)
	\qquad \text{as } j\to\infty.
	\]
	Passing to the limit in
	\[
	|Dv_{h_j}(x)-p_{k,h_j}|
	\le C|x|^{-\varepsilon},
	\]
	we obtain
	\[
	|Du_k(x)-p_k|
	\le C|x|^{-\varepsilon},
	\qquad |x|\ge R_1.
	\]
	
	The vector \(p_k\) is independent of the chosen subsequence. Indeed, if another
	subsequence gave a limit \(p_k'\), then both \(p_k\) and \(p_k'\) would satisfy
	\[
	|Du_k(x)-p_k|\le C|x|^{-\varepsilon},
	\qquad
	|Du_k(x)-p_k'|\le C|x|^{-\varepsilon}.
	\]
	Therefore
	\[
	|p_k-p_k'|
	\le
	|p_k-Du_k(x)|+|Du_k(x)-p_k'|
	\le
	C|x|^{-\varepsilon}.
	\]
	Letting \(|x|\to\infty\), we obtain \(p_k=p_k'\). Hence the full family
	\(p_{k,h}\) converges to \(p_k\), and
	\[
	|Du_k(x)-p_k|
	\le C|x|^{-\varepsilon},
	\qquad |x|\ge R_1.
	\]
	
	Since \(k=1,2\) was arbitrary, define the constant matrix \(A=(A_{ik})\) by
	\[
	A_{ik}:=(p_k)_i.
	\]
	Then
	\[
	|D^2u(x)-A|\le C|x|^{-\varepsilon},
	\qquad |x|\ge R_1.
	\]
  In particular,
	\[
	D^2u(x)\to A
	\qquad \text{as } |x|\to\infty.
	\]
	Since \(f(x)\to0\) and \(G(D^2u(x))=f(x)\), passing to the limit gives
	\[
	G(A)=0.
	\]
	
	The growth estimate \eqref{gru} follows from the fundamental theorem of calculus.
	This completes the proof.
\end{proof}

We next derive higher-order asymptotic behavior for the supercritical Lagrangian mean curvature equation.

\begin{lem}\label{ab1}
	Assume the hypotheses of Theorem~\ref{D2ue}. Suppose that there exists $A\in \mathcal{A}$ such that
	\begin{equation}\label{D2uee}
		\left|u(x)-\frac{1}{2}x^\mathsf{T}Ax\right|\leq C_1|x|^{2-\varepsilon}\quad |x|\geq R_0,
	\end{equation}
	for some constants $\varepsilon$ and $C_1$. Assume also that $f$ satisfies \eqref{Ducond} for some constant $\beta>0$. Set
	\[
	w(x):=u(x)-\frac{1}{2}x^{\mathsf{T}}Ax.
	\]
	Then there exist constants $C(\theta,R_0,\varepsilon,C_1,A,\beta)>0$ and $R_1(\theta,R_0,\varepsilon,A,C_1,\beta)>R_0$ such that for any $\alpha\in(0,1)$,
	\begin{equation}\label{Dkuee}
		|D^k w(x)| \leq C|x|^{2-k-\varepsilon_\beta} \; \text{and} \; \frac{|D^{2}w(x_1)-D^{2}w(x_2)|}{|x_1-x_2|^\alpha} \leq C|x_1|^{-\varepsilon_\beta-\alpha}
	\end{equation}
	hold for all $|x|>R_1$, $k=0,1,2$, and $|x_1|>R_1$, $x_2\in B_{|x_1|/2}(x_1)$. Here $\varepsilon_\beta=\min\{\varepsilon,\beta\}$.
\end{lem}
\begin{proof}
	For sufficiently large \( R := |x| > 1 \), set
	\[
	u_R(y) := \left( \frac{4}{R} \right)^2 u \left( x + \frac{R}{4} y \right)
	\]
	and
	\[
	w_R(y) := \left( \frac{4}{R} \right)^2 w \left( x + \frac{R}{4} y \right)
	\]
	in \( B_2 \).
	
	It follows from \eqref{D2uee} that
	\[
	\max_{y \in B_2} |u_R(y)| \leq \frac{16}{R^2} \max_{z \in B_{3R} \setminus B_{R/2}} |u(z)| \leq C
	\]
	for some \( C > 0 \) independent of \( R \) and
	\[
	\|w_R\|_{L^\infty(B_2)} \leq \frac{16}{R^2} \left\| u - \frac{1}{2}x^{\mathsf{T}}Ax \right\|_{L^\infty(B_{3R} \setminus B_{R/2})} = O(R^{-\varepsilon}).
	\]
	
	Then
	\[
	F(D^2 u_R(y)) = F \left( D^2 u \left( x + \frac{R}{4} y \right) \right) =\theta+ f \left( x + \frac{R}{4} y \right) =: f_R(y).
	\]
	
	By a direct computation and condition~\eqref{Ducond},
	\[
	\| f_R - \theta \|_{C^{0,1}(B_2)} \leq CR^{-\beta}
	\]
	for some positive constant $C$ independent of $R$. Hence, by the Hessian estimate in \cite{Z2025}, there exists $C$ independent of $R$ such that 
	\[\|u_{R}\|_{C^2(B_1)}\leq C\quad \text{and}\quad \|w_{R}\|_{C^2(B_1)}\leq C.\]
	Consequently, $F$ is uniformly elliptic along all $u_R$ and is concave in the level-set sense \cite{CPW2017}. By the Evans--Krylov estimate and the Schauder theory, for any $0 < \alpha < 1$, we have
	\[\|u_{R}\|_{C^{2,\alpha}(B_{\frac{1}{2}})}\leq C.\]
	
	The difference between \eqref{eq} and $F(A) = \theta$ gives
	\[\tilde{a}_{ij}^R \partial_{ij} w_R = f_{R}(y) - \theta=O\left(R^{-\beta}\right)\]
	where $\tilde{a}_{ij}^R(y) = \int_0^1 F_{M_{ij}}\left(tD^2w_R(y) + A\right)dt$. 
	
	Since $\tilde{a}_{ij}^R$ and $f_R$ are bounded in the $C^\alpha$ norm and $\frac{1}{C}\leq \tilde{a}_{ij}^R\leq C$, the Schauder estimate gives
	\begin{equation*}
		\|w_R\|_{C^{2,\alpha}(B_{\frac{1}{4}})}\leq C\left(\|w_R\|_{L^\infty(B_{\frac{1}{2}})}+\|f_{R}-\theta\|_{C^\alpha(B_{\frac{1}{2}})}\right)\leq CR^{-\varepsilon_\beta}.
	\end{equation*}
	Scaling back gives \eqref{Dkuee}.
\end{proof}

The following lemma gives the asymptotic behavior of harmonic functions at infinity; we refer the reader to \cite{BLZ2015,LL24,QB20252n} for details.
\begin{lem}\label{hm2}
	Let $u$ be a smooth solution of
	\[
\Delta u(x)=0,\qquad x\in \mathbb{R}^n\setminus\overline{B}_1(0),
\]
	and suppose that $u(x)=O(|x|^{\xi})$ as $|x|\to\infty$ for some $\xi<2$. Then
	\begin{itemize}
		\item $ n=2$: 
		 \[u(x) = b^\mathsf{T} x + d \ln |x| + c + O_{k}(|x|^{-1}) \text{ as } |x| \to \infty,\]
		\item $ n\geq 3$:
		 \[u(x) = b^\mathsf{T} x + c + O_{k}(|x|^{2-n}) \text{ as } |x| \to \infty,\]
	\end{itemize}
	where $b \in \mathbb{R}^n$, $c,d \in \mathbb{R}$, and the estimates hold for every $k\in \mathbb{N}$. Moreover, $b=0$ provided $0 < \xi < 1$, $b=d=0$ provided $\xi=0$, and $b=c=d=0$ provided $-1<\xi < 0$.
\end{lem}
	
Using Lemma~\ref{ab1}, we next improve the decay rate.	
The preliminary convergence input is taken from \cite{BJ2026}, and the following fractional-Laplacian argument is similar to \cite[Lemma~3.2]{QB2025}; we state the result as follows:
\begin{lem}\label{ab2}
	Assume the same hypotheses as in Lemma~\ref{ab1}, and let $R_1$ be the large constant
	determined in the proof of Lemma~\ref{ab1}. If, in addition, $2\varepsilon < 1$, then for $n \geq 2$ and $u\in C_{\mathrm{loc}}^{2,\alpha}(\mathbb{R}^n\setminus \overline{B}_{R_1})$, the estimates
	\begin{equation*}
		|D^kw(x)|\leq C|x|^{2-2\varepsilon-k}\;\text{and} \; 	\frac{|D^{2}w(x_1)-D^{2}w(x_2)|}{|x_1-x_2|^\alpha}\leq C|x_1|^{-2\varepsilon-\alpha}
	\end{equation*}
	hold for all $|x|>2R_1$, $k = 0,1,2$, and $|x_1| > 2R_1$, $x_2 \in B_{|x_1|/2}(x_1)$. 
\end{lem}

We now prove Theorem~\ref{D2ue}. The argument is divided into four steps; the nonlocal part of the proof is kept explicit because it is the key input for the improved rate.

\textbf{Step 1.} Consider the new potential \eqref{npeq}, which satisfies a uniformly elliptic equation with $F$ concave in the level-set sense. By Lemma~\ref{abf11} and Theorems~\ref{prop:quantitative_hessian} and~\ref{wdecay}, there exists a symmetric positive definite matrix $\tilde{A}$ with $F(\tilde{A}) = \tilde{\theta}$ such that
\[
D_{\tilde{x}}^2 \tilde{u}(\tilde{x}) = \tilde{A} + O(|\tilde{x}|^{-\varepsilon'}) \quad \text{as } |\tilde{x}| \to +\infty,
\]
where $\varepsilon' \in (0,1)$ is a constant depending only on $\theta$, $\beta$ and $\delta$.

Consequently, Lemmas~\ref{ab1} and~\ref{ab2} both apply for $\tilde{u}$. To obtain the limit of $D^2 u$, Proposition~\ref{rtp}(iii) implies that we need to determine the asymptotic behavior of $D_{\tilde{x}} \tilde{u}(\tilde{x})$.

In what follows, for simplicity, we write $u$ instead of $\tilde{u}$.

\textbf{Step 2. Determining the linear term.}

We may repeat Lemma~\ref{ab2} $n_0$ times so that $2^{n_0} \varepsilon < 1$ and $2^{n_0+1} \varepsilon > 1$ (after decreasing \(\varepsilon\), if necessary), provided that $\beta \geq 1$. Set $\varepsilon_1 := 2^{n_0} \varepsilon$; clearly $1 < 2\varepsilon_1 < 2$. Then, for all $|x| > 2^{2n_0} R_1$, $k = 0,1,2$, and for $|x_1| > 2^{2n_0} R_1$, $x_2 \in B_{|x_1|/2}(x_1)$, we have
\begin{equation}\label{abwn}
	|D^k w(x)| \le C |x|^{2 - \varepsilon_1 - k}, \qquad
	\frac{|D^2 w(x_1) - D^2 w(x_2)|}{|x_1 - x_2|^\alpha} \le C |x_1|^{-\varepsilon_1 - \alpha}.
\end{equation}

In particular, for $k = 1$ we obtain the desired asymptotic behavior of $D_{\tilde{x}} \tilde{u}(\tilde{x})$. Hence, by Proposition~\ref{rtp}(iii), there exists $A \in \mathcal{A}$ such that
\[
D^2 u \to A \quad \text{as } |x| \to +\infty.
\]
Moreover, $|D^2 u| \le C(\theta, f, u)$ is bounded. Applying Theorems~\ref{prop:quantitative_hessian} and~\ref{wdecay} again yields
\[
D^2 u(x) = A + O(|x|^{-\varepsilon}) \quad \text{as } |x| \to +\infty,
\]
where $\varepsilon \in (0,1)$ is a constant depending only on $\theta$, $\beta$, and $u$.

Notice that Lemmas~\ref{ab1}--\ref{ab2}, and the argument in Step~2 also apply to $u$; we will continue to refine the asymptotic expansion.

Following the notation of \cite{QB2025}, let $\tilde{a}_{ij}(x) = \int_0^1 F_{M_{ij}}\bigl(t D^2 w(x) + A\bigr) \, dt$, 
\[
\begin{aligned}
	L(x) &:= (-\Delta)^s f(x) - \partial_{ij} w(x) (-\Delta)^s \tilde{a}_{ij}(x) - \bigl(\tilde{a}_{ij}(x) - \delta_{ij}\bigr) (-\Delta)^s (\partial_{ij} w)(x) \\
	&\quad + c_{n,s} \int_{\mathbb{R}^n} \frac{\bigl( \tilde{a}_{ij}(x) - \tilde{a}_{ij}(y) \bigr) \bigl( \partial_{ij} w(x) - \partial_{ij} w(y) \bigr)}{|x - y|^{n+2s}} \, dy.
\end{aligned}
\]
and 
\[
H(x) := c_{n,-s} \int_{\mathbb{R}^n} \frac{L(y)}{|x - y|^{n-2s}} \, dy.
\]
Then we obtain
\[
(-\Delta)^s (\Delta w) = L, \quad L \in C^{\alpha - 2s}(\mathbb{R}^n),
\]
where $s\in(0,\alpha/2)$. Since $\beta>2$ and $2\varepsilon_1<2$, we obtain
\[
|L(x)| \le C |x|^{-2s - 2\varepsilon_1}, \quad |x| > 2^{2n_0+1} R_1,
\]
and $H$ is continuous on $\mathbb{R}^2$ satisfying
\[
(-\Delta)^s H(x) = L(x), \qquad |H(x)| \le C |x|^{-2\varepsilon_1}, \quad |x| > 2^{2n_0+1} R_1.
\]

Consequently,
\[
(-\Delta)^s (\Delta w - H) = 0, \quad |x| > 2^{2n_0+1} R_1.
\]
In view of \cite[Lemma~2.1]{QB20252n},
\[
|\Delta w(x) - H(x)| \le C |x|^{2s-n}, \quad |x| > 2^{2n_0+1} R_1.
\]
Since $1 < 2\varepsilon_1 < 2$ and $1 < n - 2s < n$, we can find a constant $\nu \in (0,1)$ such that
\[
|\Delta w(x)| \le C |x|^{-1 - \nu}, \quad |x| > 2^{2n_0+1} R_1.
\]

By \cite[Lemma~1]{LB2023ANS}, there exists $\overline{w}$ satisfying $\Delta \overline{w}(x) = \Delta w(x)$ for $|x| \ge 2^{2n_0+1} R_1$ such that
\[
|\overline{w}(x)| \le C |x|^{1 - \nu}, \quad |x| > 2^{2n_0+1} R_1.
\]
Since $w-\overline{w}$ is harmonic and $|w(x) - \overline{w}(x)| \le C |x|^{2 - \varepsilon_1}$ for $|x| > 2^{2n_0+1} R_1$, Lemma~\ref{hm2} then implies the existence of $b \in \mathbb{R}^2$ such that
\[
w(x) - \overline{w}(x) = b^\mathsf{T} x + O(\ln |x|), \quad |x| \to +\infty.
\]
if $n=2$. When $n\geq 3$, there exists $b \in \mathbb{R}^n$ such that
\[
w(x) - \overline{w}(x) = b^\mathsf{T} x + O(1), \quad |x| \to +\infty.
\]

Consequently,
\[
|w(x) - b^\mathsf{T} x| \le C |x|^{1 - \nu}, \quad |x| > 2^{2n_0+1} R_1.
\]

\textbf{Step 3. Determining the logarithmic term and constant term.}

Set
\[
w_1(x) = w(x) - b^\mathsf{T} x.
\]
By Step~2 and Lemma~\ref{ab1}, we have $|D^k w_1(x)| \le C |x|^{1-\nu - k}$ for $k = 0,1,2$.

The equation for $w_1$ can then be written as
\begin{equation}\label{w1}
	\Delta w_1 = f - (\tilde{a}_{ij} - \delta_{ij}) (w_1)_{ij} = O(|x|^{-\beta}) + O(|x|^{-2-2\nu}).
\end{equation}

When $n=2$, we define
\[
H_1^2(x) := \frac{1}{2\pi} \int_{\mathbb{R}^2 \setminus B_{2^{2n_0+1}R_1}} \bigl(\ln|x-y| - \ln|x|\bigr) \Delta w_1(y) \, dy.
\]
Then $w_1 - H_1^2$ is harmonic on $\mathbb{R}^2 \setminus B_{2^{2n_0+1}R_1}$. Moreover, by \cite[Lemma~2.2]{QB20252n}, we have
\begin{equation}\label{h1}
	|H_1^2(x)| \le C \bigl( |x|^{-\min\{1,\beta-2\} + \nu_1} + |x|^{-2\nu + \nu_1} \bigr)
\end{equation}
for any arbitrarily small $\nu_1 > 0$, and consequently
\[
|w_1(x) - H_1^2(x)| = O(|x|^{1-\nu}).
\]
When $n\geq 3$, we construct $L_1$ and $H_1^n$ in the same way as $L$ and $H$ in Step~2, with $w$ replaced by $w_1$. Then, 
\[|L_1(x)|\leq C|x|^{-2s-\beta}+C|x|^{-2s-2-2\nu},\quad |x|>2^{2n_0+1}R_1.\]
Since $n\geq 3$, \cite[Corollary~2.9]{QB2025} implies $H_1^n$ admits the decay estimate
\begin{equation}\label{hn}
	|H_1^n(x)|\leq C|x|^{-\min \{\beta, n-2s\}}+C|x|^{-\min\{2+2\nu,n-2s\}}\leq C|x|^{-2-\nu_1},\quad |x|>2^{2n_0+2}R_1
\end{equation}
and consequently, 
\[
|w_1(x) - H_1^n(x)| = O(|x|^{1-\nu}).
\]

Now, for $n=2$, it follows from Lemma~\ref{hm2} that there exist $c, d \in \mathbb{R}$ such that
\begin{equation*}
	w_1(x) - H_1^2(x) = d \ln|x| + c + O(|x|^{-1}).
\end{equation*}
Combining this with \eqref{h1}, we obtain
\[
w_1(x) = d \ln|x| + c + O(|x|^{-\nu_2}), \quad |x| \to +\infty,
\]
for some \(\nu_2>0\). Decreasing \(\nu_2\), if necessary, we may assume $\nu_2 < \beta - 2$.

When $n\geq 3$, there exists $c \in \mathbb{R}$ such that
\begin{equation*}
	w_1(x) - H_1^n(x) = c + O(|x|^{2-n}).
\end{equation*}
Combining this with \eqref{hn}, we obtain
\[
w_1(x) = c + O(|x|^{-\nu_3}), \quad |x| \to +\infty,
\]
for some \(\nu_3>0\). Decreasing \(\nu_3\), if necessary, we may assume $\nu_3 < 1$.

\textbf{Step 4. Determining the remainder term.}

First consider the $n=2$ case. Define
\[
w_2^2(x) = w_1(x) - d \ln|x| - c.
\]

From Step~3 we already have
\[
|w_2^2(x)| \le C |x|^{-\nu_2}, \quad |x| > 2^{2n_0+2}R_1.
\]
Using Lemma~\ref{ab1} together with Schauder estimates, we obtain
\[
|D^k w_2^2(x)| \le C |x|^{-\nu_2 - k}, \quad |x| > 2^{2n_0+3}R_1, \quad k = 0,1,2.
\]
Thus the equation for $w_2^2$ can be expressed as
\[
\Delta w_2^2(x) = f(x) - (\hat{a}_{ij} - \delta_{ij})(w_2^2)_{ij} = O(|x|^{-\beta}) + O(|x|^{-4-\nu_2}),
\]
where $\hat{a}_{ij}(x) = \int_0^1 F_{M_{ij}}\bigl(t(D^2 w_2^2(x) + D^2(d \ln|x|)) + A\bigr) \, dt$.

Applying \cite[Lemma~1]{LB2023}, there exists $H_2^2$ such that $\Delta H_2^2 = \Delta w_2^2$ for $|x| > 2^{2n_0+3}R_1$ and
\[
H_2^2(x) = 
\begin{cases}
	O(|x|^{-\min\{\beta-2,\,2+\nu_2\}}), & \beta \neq 3,4, \\[4pt]
	O(|x|^{2-\beta} \ln|x|), & \beta = 3,4,
\end{cases}
\quad |x| \to +\infty.
\]
Since $w_2^2 - H_2^2$ is harmonic on $\mathbb{R}^2 \setminus B_{2^{2n_0+3}R_1}$ and $w_2^2 - H_2^2 = O(|x|^{-\nu_4})$ for some $\nu_4 > 0$, Lemma~\ref{hm2} yields
\[
w_2^2(x) - H_2^2(x) = O(|x|^{-1}).
\]

Consequently, we obtain
\[
w_2^2(x) = 
\begin{cases}
	O(|x|^{-\min\{\beta-2,\,1\}}), & \beta \neq 3, \\[4pt]
	O(|x|^{-1} \ln|x|), & \beta = 3,
\end{cases}
\quad |x| \to +\infty.
\]
We may then argue as in Lemma~\ref{ab1} to obtain estimates on $D^k w_2^2$ for $k = 1,2$. Moreover, for any $\alpha \in (0,1)$,
\[
\limsup_{|x| \to +\infty} |x|^{\min\{\beta-2,\,1\} + \alpha} (\ln|x|)^{-\mu_1} [D^2 w_2^2]_{C^{0,\alpha}(\overline{B_{|x|/2}(x)})} < +\infty,
\]
where $\mu_1$ is defined in Theorem~\ref{D2ueo}.

When $n\geq 3$, we define
\[w_2^n(x) = w_1(x) - c.\]
Similarly, by Step~3, the equation for $w_2^n$ can be expressed as
\[
\Delta w_2^n(x) = f(x) - (\overline{a}_{ij} - \delta_{ij})(w_2^n)_{ij} = O(|x|^{-\beta}) + O(|x|^{-4-\nu_3}),
\]
where $\overline{a}_{ij}(x) = \int_0^1 F_{M_{ij}}\bigl(t(D^2 w_2^n(x)) + A\bigr) \, dt$.

Applying \cite[Lemma~1]{LB2023}, there exists $H_2^n$ such that $\Delta H_2^n = \Delta w_2^n$ for $|x| > 2^{2n_0+3}R_1$ and
\[
H_2^n(x) = 
\begin{cases}
	O(|x|^{-\min\{\beta-2,\,2+\nu_3\}}), & \beta \neq n, \\
	O(|x|^{2-n} \ln|x|)+O(|x|^{-2-2\nu_3}), & \beta = n,
\end{cases}
\quad |x| \to +\infty.
\]
Since $w_2^n - H_2^n$ is harmonic on $\mathbb{R}^n \setminus B_{2^{2n_0+3}R_1}$ and $w_2^n - H_2^n = O(|x|^{-\nu_5})$ for some $\nu_5 > 0$, Lemma~\ref{hm2} yields
\[
w_2^n(x) - H_2^n(x) = O(|x|^{2-n}).
\]

Consequently, we obtain
\[
w_2^n(x) = 
\begin{cases}
	O(|x|^{2-\min\{\beta,\,n\}})+O(|x|^{-2-2\nu_3}), & \beta \neq n, \\
		O(|x|^{2-n} \ln|x|)+O(|x|^{-2-2\nu_3}), & \beta = n,
\end{cases}
\quad |x| \to +\infty.
\]
If $|x|^{-2-2\nu_3}>|x|^{2-n}+|x|^{2-\beta}$ (or $|x|^{-2-2\nu_3}>|x|^{2-n} \ln|x|$), we can repeat Step 4 finitely many times to remove the $|x|^{-2-2\nu_3}$ from the above estimate.

We may then argue as in Lemma~\ref{ab1} to obtain estimates on $D^k w_2^n$ for $k = 1,2$. Moreover, for any $\alpha \in (0,1)$,
\[
\limsup_{|x| \to +\infty} |x|^{\min\{\beta,\,n\}-2 + \alpha} (\ln|x|)^{-\mu_2} [D^2 w_2^n]_{C^{0,\alpha}(\overline{B_{|x|/2}(x)})} < +\infty,
\]
where $\mu_2$ is defined in Theorem~\ref{D2ueo}.

The calculation of the constant $d$ in Theorem~\ref{D2ue} is consistent with that in \cite{BJ2026}.


This completes the proof of the fast convergence parts of Theorems~\ref{D2ue} and~\ref{D2uehd} under the assumption $F_{M_{ij}}(A) = \delta_{ij}$. If $F_{M_{ij}}(A) \neq \delta_{ij}$, let $P = \bigl( [\sqrt{F_{M_{ij}}(A)}]_{2\times 2} \bigr)^{-1} = \sqrt{I + A^2}$ and set $\hat{x} = P x$. The argument above yields the desired conclusion in the $\hat{x}$-coordinates; transforming back to $x$ gives the statement of Theorem~\ref{D2ue}. \qed

\section{\texorpdfstring{The slow convergence case $0<\beta\le 2$}{The slow convergence case 0 < beta <= 2}}\label{sec4}
In this section, we treat the slow convergence regime $0<\beta\le2$.

Note that Theorems~\ref{prop:quantitative_hessian} and~\ref{wdecay}, and Lemma~\ref{ab1} remain valid under the slow convergence assumption on $f$. Moreover, there exists a positive constant $m_0$ such that $2^{m_0}\varepsilon < \min\{\beta,1\} < 2^{m_0+1}\varepsilon$. Lemma~\ref{ab2} also holds, and for any $\alpha \in (0,1)$,
\begin{equation}\label{dkw}
		|D^k w(x)| \le C |x|^{2 - \varepsilon_1 - k}, \quad
		\frac{|D^2 w(x_1) - D^2 w(x_2)|}{|x_1 - x_2|^\alpha} \le C |x_1|^{-\varepsilon_1 - \alpha},
\end{equation}
for all $|x| > 2^{2m_0} R_1$, $k = 0,1,2$, and $|x_1| > 2^{2m_0} R_1$, $x_2 \in B_{|x_1|/2}(x_1)$, where $\varepsilon_1 = 2^{m_0} \varepsilon$.

Since $\min\{\beta,1\} < 2\varepsilon_1 < 2\min\{\beta,1\}$, for fixed $\varepsilon_1$ and $\beta \le 2$, we may choose $s > 0$ sufficiently small such that
\begin{equation}\label{2sbe}
	2s + \min\{\beta, 2\varepsilon_1\} < 2.
\end{equation}

Let $L$ and $H$ be as in Step~2 of Section~\ref{sec3}. Following the proof in Step~2 and using $2\varepsilon_1 < 2$, we find that in two dimensions, 
\[
|L(x)| \le C |x|^{-2s - \min\{\beta, 2\varepsilon_1\}}, \quad |x| > 2^{2m_0+1} R_1.
\]

Therefore, \cite[Corollary~2.9]{QB2025} yields
\[
|H(x)| \le C |x|^{-\min\{\beta, 2\varepsilon_1\}}, \quad |x| > 2^{2m_0+2} R_1,
\]
and consequently, for $n=2$,
\begin{equation*}
	|\Delta w(x)| \le 
	\begin{cases} 
		C |x|^{-\beta}, & 0 < \beta \le 1, \\
		C |x|^{-\min\{\beta, 2\varepsilon_1, 2-2s\}}, & 1 < \beta \le 2,
	\end{cases}
	\quad |x| > 2^{2m_0+2} R_1.
\end{equation*}
When \(0<\beta\le1\), \cite[Lemma~1]{LB2023ANS} implies that there exists $\overline{w}$ satisfying $\Delta \overline{w}(x) = \Delta w(x)$ for $|x| \ge 2^{2m_0+1} R_1$ such that
\[
|D^k \overline{w}(x)| \le 
\begin{cases} 
	C |x|^{2 - k - \beta}, & 0 < \beta < 1, \\[4pt]
	C |x|^{1 - k} \ln|x|, & \beta = 1,
\end{cases}
\quad |x| > 2^{2m_0+2} R_1,
\]
for $k = 0,1,2$.

Since $w-\overline{w}$ is harmonic and $|w(x) - \overline{w}(x)| \le C |x|^{2 - \varepsilon_1}$ for $|x| > 2^{2m_0+1} R_1$, Lemma~\ref{hm2} then implies the existence of $b \in \mathbb{R}^2$ such that
\[
w(x) - \overline{w}(x) = b^\mathsf{T} x + O(\ln |x|), \quad |x| \to +\infty.
\]
Combining this with Lemma~\ref{ab1}, we obtain the decay estimate
\[
|D^k w(x)| \le 
\begin{cases} 
	C |x|^{2 - k - \beta}, & 0 < \beta < 1, \\[4pt]
	C |x|^{1 - k} \ln|x|, & \beta = 1,
\end{cases}
\quad |x| > 2^{2m_0+2} R_1,
\]
for $k = 0,1,2$. Moreover, by Lemma~\ref{ab1}, for any $\alpha \in (0,1)$,
\[
\limsup_{|x| \to +\infty} |x|^{\beta + \alpha} (\ln|x|)^{-[\beta]} [D^2 w]_{C^{0,\alpha}(\overline{B_{|x|/2}(x)})} < +\infty.
\]

When $1 < \beta \le 2$, since $1 < 2\varepsilon_1 < 2$, the argument in Step~2 of Section~\ref{sec3} gives a vector $b \in \mathbb{R}^2$ such that
\[
w(x) = b^\mathsf{T} x + o(|x|) \quad \text{as } |x| \to +\infty.
\]
Define $w_1$, $L_1$, and $H_1$ as in the proof of Theorem~\ref{D2ue}. Then
\[
(-\Delta)^s (\Delta w_1) = L_1.
\]

From \eqref{dkw} and \eqref{2sbe}, we obtain the estimates
\[
|L_1(x)| \le C |x|^{-2s - \min\{2\varepsilon_1, \beta\}}, \quad |x| > 2^{2m_0+1} R_1,
\]
and
\[
|H_1(x)| \le C |x|^{-\min\{2\varepsilon_1, \beta\}}, \quad |x| > 2^{2m_0+2} R_1.
\]

Together with
\[
|\Delta w_1(x) - H_1(x)| \le C |x|^{2s-2}, \quad |x| > 2^{2m_0+2} R_1,
\]
we deduce that $|\Delta w_1(x)| \le C |x|^{-\min\{2\varepsilon_1, \beta\}}$ for $|x| > 2^{2m_0+2} R_1$. Since $\min\{2\varepsilon_1, \beta\} \in (1,2)$, by \cite[Lemma~1]{LB2023ANS} and Lemma~\ref{ab1}, we have $|D^k w_1(x)| \le C |x|^{2 - k - \min\{2\varepsilon_1, \beta\}}$ for $k = 0,1,2$. As in \eqref{w1}, we compute
\[
\Delta w_1 = f - (\tilde{a}_{ij} - \delta_{ij}) (w_1)_{ij} = O(|x|^{-\beta}) + O(|x|^{ - 2\min\{2\varepsilon_1, \beta\}}) = O(|x|^{-\beta}).
\]

Applying \cite[Lemma~1]{LB2023ANS} again, we find a function $\phi$ such that
\[
\Delta \phi(x) = \Delta w_1(x) \quad \text{for } |x| > 2^{2m_0+2} R_1,
\]
and
\[
|\phi(x)| \le 
\begin{cases} 
	C |x|^{2-\beta}, & 1 < \beta < 2, \\[4pt]
	C (\ln|x|)^2, & \beta = 2.
\end{cases}
\]

Since $w_1 - \phi$ is harmonic for $|x| > 2^{2m_0+2} R_1$ and $|w_1(x) - \phi(x)| \le C |x|^{\nu_5}$ for some $0 < \nu_5 < 1$, Lemma~\ref{hm2} yields
\[
w_1(x) - \phi(x) = O(\ln|x|) \quad \text{as } |x| \to +\infty.
\]

Combining these estimates with Lemma~\ref{ab1}, we obtain
\[
|D^k w_1(x)| \le 
\begin{cases} 
	C |x|^{2-\beta - k}, & 1 < \beta < 2, \\[4pt]
	C |x|^{-k} (\ln|x|)^2, & \beta = 2,
\end{cases}
\quad k = 0,1,2.
\]
Moreover, by Lemma~\ref{ab1}, for any $\alpha \in (0,1)$,
\[
\limsup_{|x| \to +\infty} |x|^{\beta + \alpha} (\ln|x|)^{-2([\beta]-1)} [D^2 w_1]_{C^{0,\alpha}(\overline{B_{|x|/2}(x)})} < +\infty. 
\]

For $n\geq 3$, since $n-2s>2$,
\begin{equation*}
	|\Delta w(x)| \le 
	\begin{cases} 
		C |x|^{-\beta}, & 0 < \beta \le 1, \\
		C |x|^{-\min\{\beta, 2\varepsilon_1\}}, & 1 < \beta \le 2,
	\end{cases}
	\quad |x| > 2^{2m_0+2} R_1.
\end{equation*}
When $0<\beta\leq 1$, \cite[Lemma~1]{LB2023ANS} implies that there exists $\overline{w}$ satisfying $\Delta \overline{w}(x) = \Delta w(x)$ for $|x| \ge 2^{2m_0+1} R_1$ such that
\[
|D^k \overline{w}(x)| \le 
\begin{cases} 
	C |x|^{2 - k - \beta}, & 0 < \beta < 1, \\[4pt]
	C |x|^{1 - k} \ln|x|, & \beta = 1,
\end{cases}
\quad |x| > 2^{2m_0+2} R_1,
\]
for $k = 0,1,2$.

Since $w-\overline{w}$ is harmonic and $|w(x) - \overline{w}(x)| \le C |x|^{2 - \varepsilon_1}$ for $|x| > 2^{2m_0+1} R_1$, Lemma~\ref{hm2} then implies the existence of $b \in \mathbb{R}^n$ such that
\[
w(x) - \overline{w}(x) = b^\mathsf{T} x + o(|x|), \quad |x| \to +\infty.
\]
Combining this with Lemma~\ref{ab1}, we obtain the decay estimate
\[
|D^k w(x)| \le 
\begin{cases} 
	C |x|^{2 - k - \beta}, & 0 < \beta < 1, \\[4pt]
	C |x|^{1 - k} \ln|x|, & \beta = 1,
\end{cases}
\quad |x| > 2^{2m_0+2} R_1,
\]
for $k = 0,1,2$. Moreover, by Lemma~\ref{ab1}, for any $\alpha \in (0,1)$,
\[
\limsup_{|x| \to +\infty} |x|^{\beta + \alpha} (\ln|x|)^{-[\beta]} [D^2 w]_{C^{0,\alpha}(\overline{B_{|x|/2}(x)})} < +\infty.
\]

When $1 < \beta \le 2$, since $1 < 2\varepsilon_1 < 2$, the argument in Step~2 of Section~\ref{sec3} gives a vector $b \in \mathbb{R}^n$ such that
\[
w(x) = b^\mathsf{T} x + o(|x|) \quad \text{as } |x| \to +\infty.
\]
Define $w_1$, $L_1$, and $H_1$ as in the proof of Theorem~\ref{D2ue}. Then
\[
(-\Delta)^s (\Delta w_1) = L_1.
\]

From \eqref{dkw} and \eqref{2sbe}, we obtain the estimates
\[
|L_1(x)| \le C |x|^{-2s - \min\{2\varepsilon_1, \beta\}}, \quad |x| > 2^{2m_0+1} R_1,
\]
and
\[
|H_1(x)| \le C |x|^{-\min\{2\varepsilon_1, \beta\}}, \quad |x| > 2^{2m_0+2} R_1.
\]

Together with
\[
|\Delta w_1(x) - H_1(x)| \le C |x|^{2s-n}, \quad |x| > 2^{2m_0+2} R_1,
\]
we deduce that $|\Delta w_1(x)| \le C |x|^{-\min\{2\varepsilon_1, \beta\}}$ for $|x| > 2^{2m_0+2} R_1$. Since $\min\{2\varepsilon_1, \beta\} \in (1,2)$, by \cite[Lemma~1]{LB2023ANS} and Lemma~\ref{ab1}, we have $|D^k w_1(x)| \le C |x|^{2 - k - \min\{2\varepsilon_1, \beta\}}$ for $k = 0,1,2$. As in \eqref{w1}, we compute
\[
\Delta w_1 = f - (\tilde{a}_{ij} - \delta_{ij}) (w_1)_{ij} = O(|x|^{-\beta}) + O(|x|^{ - 2\min\{2\varepsilon_1, \beta\}}) = O(|x|^{-\beta}).
\]

Applying \cite[Lemma~1]{LB2023ANS} again, we find a function $\phi$ such that
\[
\Delta \phi(x) = \Delta w_1(x) \quad \text{for } |x| > 2^{2m_0+2} R_1,
\]
and
\[
|\phi(x)| \le 
\begin{cases} 
	C |x|^{2-\beta}, & 1 < \beta < 2, \\[4pt]
	C \ln|x|, & \beta = 2.
\end{cases}
\]

Since $w_1 - \phi$ is harmonic for $|x| > 2^{2m_0+2} R_1$ and $|w_1(x) - \phi(x)| \le C |x|^{\nu_6}$ for some $0 < \nu_6 < 1$, Lemma~\ref{hm2} yields
\[
w_1(x) - \phi(x) = O(\ln|x|) \quad \text{as } |x| \to +\infty.
\]

Combining these estimates with Lemma~\ref{ab1}, we obtain
\[
|D^k w_1(x)| \le 
\begin{cases} 
	C |x|^{2-\beta - k}, & 1 < \beta < 2, \\[4pt]
	C |x|^{-k} \ln|x|, & \beta = 2,
\end{cases}
\quad k = 0,1,2.
\]
Moreover, by Lemma~\ref{ab1}, for any $\alpha \in (0,1)$,
\[
\limsup_{|x| \to +\infty} |x|^{\beta + \alpha} (\ln|x|)^{-[\beta]+1} [D^2 w_1]_{C^{0,\alpha}(\overline{B_{|x|/2}(x)})} < +\infty. 
\]
This completes the proof of the slow convergence parts of Theorems~\ref{D2ue} and~\ref{D2uehd}.
\qed

\end{document}